\documentclass{article}

\usepackage{amsmath}
\usepackage{amsthm}
\usepackage{amsfonts}
\usepackage{amssymb}
\usepackage{setspace}

\usepackage{authblk}

\usepackage{graphicx}
\graphicspath{ {./images/} }

\usepackage{caption}
\usepackage{subcaption}

\newcommand{\Real}{\mathbb{R}}
\newcommand{\Natural}{\mathbb{N}}
\newcommand{\Grid}{\mathcal{G}}

\DeclareMathOperator{\diam}{\mathrm{diam}}
\DeclareMathOperator{\prob}{\Pr}
\DeclareMathOperator{\supp}{\mathrm{supp}}
\DeclareMathOperator{\diag}{\mathrm{diag}}

\theoremstyle{plain}
\newtheorem{thmm}{Theorem}[section]
\newtheorem{definition}{Definition}[section]
\newtheorem{corollary}[thmm]{Corollary}
\newtheorem{lem}[thmm]{Lemma}
\newtheorem{remark}{Remark}[section]
\numberwithin{equation}{section}

\providecommand{\keywords}[1]
{
  \small	
  \textbf{\textit{Keywords---}} #1
}

\begin{document}

\title{The discrete charm of iterated function systems\\ 
 \large A computer scientist's perspective on approximation\\ of IFS invariant sets and measures}

\author{Tomasz Martyn}

\affil{Warsaw University of Technology, Institute of Computer Science, \\ul. Nowowiejska 15/19, 00-665, Warsaw,
Poland\\E-mail: tomasz.martyn@pw.edu.pl}

\maketitle


\begin{abstract}
 We study invariant sets and measures generated by iterated function systems defined on countable discrete spaces that are uniform grids of a finite dimension. The discrete spaces of this type can be considered as models of spaces in which actual numerical computation takes place. In this context, we investigate the possibility of the application of the random iteration algorithm to approximate these discrete IFS invariant sets and measures. The problems concerning a discretization of hyperbolic IFSs are considered as special cases of this more general setting.
\end{abstract}


\section{The problem}\label{sect_problem}
Let $\{\Real^n; w_1, \dots, w_m \}$, $m\in \mathbb{N}$, be a \emph{hyperbolic iterated function system} (IFS) on a metric space $(\Real^n, d)$, where $d$ is a metric induced by a norm on $\mathbb{R}^n$, and the mappings $w_i$ are contractions on $(\Real^n, d)$. One can show that $\mathcal{H}(\Real^n)$, the family of all compact and nonempty subsets of $\Real^n$, when endowed with the Hausdorff metric $h$ (induced by $d$), forms a complete metric space $(\mathcal{H}(\Real^n), h)$. Moreover, if a contraction $w$ on $(\Real^n, d)$ is regarded as a set mapping $w : \mathcal{H}(\Real^n) \rightarrow \mathcal{H}(\Real^n)$, then $w$ is a contraction operator on $(\mathcal{H}(\Real^n), h)$ with the contractivity factor not greater than that of $w$ acting on $(\Real^n, d)$. This observation forms a basis for constructing the \emph{Hutchison operator} $W : \mathcal{H}(\Real^n) \rightarrow \mathcal{H}(\Real^n)$ defined by $W(E):= \bigcup_{i=1}^m w_i(E)$. The Hutchinson operator a contraction on $(\mathcal{H}(\Real^n), h)$ with the contractivity factor not exceeding the maximum of the contractivity factors of $w_i$. Finally, putting the contractivity of $W$ and the completeness of $(\mathcal{H}(\Real^n), h)$ together, by Banach's fixed-point theorem we get that $W$ possesses exactly one fixed point $A_{\infty} = W(A_{\infty})$, and moreover 
\begin{equation}\label{HutConv}
\lim_{k\rightarrow \infty} W^{\circ k}(B) = A_{\infty}   
\end{equation}
regardless of $B \in \mathcal{H}(\Real^n)$. The fixed point $A_{\infty}$ is called the \emph{attractor} of the IFS, and being an element of $\mathcal{H}(\Real^n)$ it is a nonempty and compact subset of $(\Real^n, d)$. 

An analogous argument, also based on Banach's fixed-theorem, is utilized to prove the existence and the uniqueness of the \emph{invariant measure} of the hyperbolic \emph{IFS with probabilities}, $\{K; w_1, \dots, w_m; p_1,\dots, p_m \}$, where $\sum_{i=1}^n p_i = 1$, $p_i > 0$, and $K$ is a compact subset of $(\Real^n, d)$. This time the role of the complete space is taken on by the compact (and thus complete) metric space $(\mathcal{P}(K), d_{w^*})$, where $\mathcal{P}(K)$ is the family of Borel probability measures on $K$, and $d_{w^*}$ is the Monge-Katorovich metric. The role of the contraction operator acting on $(\mathcal{P}(K), d_{w^*})$ is played by the Markov operator $T^*: \mathcal{P}(K) \rightarrow \mathcal{P}(K)$ defined by $T^*(\mu) := \sum_{i = 1}^N p_i \mu\circ w_i^{-1}$. In this setting, by the Banach fixed point theorem
\begin{equation}
    \lim_{k\rightarrow \infty} (T^*)^{\circ k}(\mu) = \pi_{\infty}
\end{equation}
no matter what $\mu \in \mathcal{P}(K)$, that is, the IFS invariant measure $\pi_{\infty} \in \mathcal{P}(K)$ is the unique attractive fixed point of the Markov operator. Moreover, one can show that $\pi_\infty$ is supported by the IFS attractor. 

All these are well-known facts from the theory of hyperbolic iterated function systems, and the appropriate proofs are long-established and can be found in quite a number of publications on IFS, to mention only such classical texts as \cite{Hutch81,Barn93}. And, as we have just pointed out, in the case of hyperbolic IFSs all these facts are founded on the contractivity of the IFS mappings along with Banach's fixed point theorem that utilizes, as one of its premises, the derived contractivity of the Hutchison and Markov operators. 

In actual implementations, however, the space upon which IFS mappings operate is always merely a discrete model of $\mathbb{R}^n$. From the point of view of computation the most basic model of the real numbers is a countable discrete space founded on floating-point arithmetic. Spaces of this kind are the natural setting within which implementation of the most popular algorithm for approximating IFS attractors and measures, the \emph{random iteration algorithm}\cite{Barn93}, generates its sequences of points. 
The problem is that contraction mappings on $\mathbb{R}^n$ can lose their contraction properties when forced to act on a countable approximation of the space, even in such simple cases as contractive similarity transformations\cite{Peru93}. In addition, as we shall see soon, numerical evidence shows that at least for affine mappings this phenomenon is not something exceptional and actually occurs more often than rarely. What is more, and might be even more surprising, this decline of contractivity caused by discretization is independent of precision, in the sense that no matter what a discretization granularity, the probability that a discrete version of an affine contraction is not a contraction remains unchanged. Therefore, we cannot get rid of this phenomenon by increasing the density of discretization. Since for an IFS to be hyperbolic all its component mappings have to be contractions, the probability that a discrete version of a hyperbolic affine IFS retains hyperbolicity decreases exponentially with the number of IFS mappings. As a consequence, actual implementations of an affine IFS are hardly ever hyperbolic even if the original IFS is so. Therefore, at least from the perspective of the standard hyperbolic IFS theoretical foundations, we cannot count on the discrete counterparts of the Hutchinson and Markov operators to be contractions in the discrete space, and thus, in such a situation, Banach's fixed-point theorem is not applicable, at least directly. In light of these facts, we can go even further and question the relevancy of outcomes supplied by the implementations of the algorithms built around the operators, asking how the outcomes are related, if at all, to the true attractor sets and measures of the original IFSs which are hyperbolic in $(\Real^n, d)$?

As for the latter question we can expect that because of the hyperbolicity featuring the original IFS, the outcomes yielded by the actual implementations of the approximation algorithms will resemble more or less the true ones, due to the orbit shadowing phenomenon known from the theory of dynamical systems. Moreover, we have to bear in mind that the hyperbolicity of IFS is a sufficient and not a necessary condition for the iteration of the Hutchinson and Markov operators to converge. Maybe, despite the potential loss of hyperbolicity by an IFS due to discretization, there are some other reasons for which the operators converge to some limiting sets and measures in a discrete space, to mention only a contractivity on average as one possibility. And if so, are the hypothetical invariant sets and measures attractive globally and thus are unique or can there be more than one invariant set or measure, some of them possessing their own basins of attraction? 

The above problems direct us to the central subject of this paper, which concerns \emph{discrete IFSs} that are not merely derivatives of hyperbolic IFSs but they are self-standing mathematical objects defined on a grid space from the very beginning. The main question around which this paper is organized is about the place of such self-standing discrete IFSs within the framework of the theory of iterated function systems. We treat the issues concerning discretization of hyperbolic IFSs as special instances of this general problem.

\section{Previous work}
The theory of hyperbolic IFSs is only the tip of the iceberg of a much more general mathematical theory that encompasses IFSs with place-dependent probabilities and composed of uncountable number of measurable maps 
that act on complete separable metric spaces (see e.g. \cite{Diac99,Sten02,Sten12}). The general IFS theory is largely built upon the theory of homogeneous discrete-time Markov chains on measurable state spaces. 
Within both theories a special position is occupied by Feller chains, which are specified by transfer operators of the form $(T f)(x) := \sum_{i=1}^N p_i(x) f (w_i(x))$ that maps $C(X)$, the space of real-valued bounded continuous functions on $(X, d)$, into iself, and $p_i : X\rightarrow [0, 1]$ are place-dependent probabilities assigned to the IFS maps $w_i$ with $\sum_i p_i(x) = 1$ at each $x\in X$. A Borel probability measure $\pi \in \mathcal{P}(X)$ is said to be invariant (w.r.t.~$T$) if $\int_x T f d\pi = \int_X f d\pi$ for all $f\in C(X)$. Now, an operator $T^{*} : \mathcal{P}(X)\rightarrow \mathcal{P}(X)$ is defined as an operator satisfying $\int_X T f d\mu = \int_X f d(T^*\mu)$ for all $f\in C(X)$, which is the adjoint of $T$, well-defined by the Riesz representation theorem\cite{Sten02,Barn88}. This operator is called the Markov operator in the books\cite{Barn93,Barn06} by Barnsley, and for an IFS with place-dependent probabilities this operator takes the form of $(T^{*} \mu)(B) = \sum_{i= 1}^N \int_{w_{i}^{-1}(B)} p_i(x) d\mu$, which for an IFS with constant probabilities reduces to the well-known definition presented in the previous section. By means of $T^*$, an IFS invariant measure can be defined as a Borel probability measure $\pi$ satisfying $T^*\pi = \pi$ and this measure is a stationary probability measure of the associated Markov (Feller) chain. The main effort of the IFS theory is concentrated on conditions that an IFS should meet for invariant measures to exist, to be unique, and to be attractive. As for the latter, a measure is (globally) attractive if $\int_X  f d((T^*)^{\circ m}\mu) \rightarrow \int_X f d\pi$ for all $\mu \in \mathcal{P}(X)$ and $f \in C(X)$, which reads as a weak convergence of $(T^*)^{\circ m}\mu$ to $\pi$. For example, a sufficient condition for an IFS to possess \emph{at least} one invariant measure is the compactness of $(X, d)$ (which can be weakened to a requirement that the chain get stuck in a compact set with probability one, see e.g.~\cite{Sten02}). This result follows from the Schauder fixed point theorem, because $\mathcal{P}(X)$ is a convex set in the space of signed Borel measures on $X$ and the compactness of $(X, d)$ implies the compactness of $\mathcal{P}(X)$ in the weak* topology, and $T^*: \mathcal{P}(X)\rightarrow \mathcal{P}(X)$ is (weak*) continuous\cite{Barn88,Barn06}. In turn, for an IFS to have a globally attractive (and thus unique) invariant measure it suffices that the IFS is composed of Lipschitz maps that are contractive on average, $\sup_{x\neq y} \sum_{i=1}^N p_i(x)d(w_i(x), w_i(y))/d(x, y) < 1$, the probability functions $p_i(.)$'s are positive and Dini-continuous, and $(X, d)$ is a locally compact separable metric space (moreover, the contractive on average condition is usually additionally weakened by the $\log$ function---see \cite{Barn88,Diac99,Sten02}).      

Since a uniform grid endowed with a metric induced by a norm is a complete, locally compact and separable space, and moreover any map on the grid is continuous, it follows that discrete IFSs are somewhere within the general theory and they have as much in common with general IFSs as countable Markov chains have with Markov chains on general measurable state spaces. Therefore, any result on general IFSs indirectly pertains also to discrete IFSs. Nevertheless, although countable Markov chains occupy a prominent place within the theory of stochastic processes, with a vast dedicated literature, the theory of general IFS seems to show little interest in discrete IFSs on their own. 
As a consequence, to a large extent IFSs on discrete spaces stay hidden in the depths of the general IFS theory. 

As already mentioned, however, a discrete space is truly the only space that we have at our disposal when it comes to algorithms and actual computation. Quite naturally, then, the area of algorithms and approximation issues is just the place where discrete IFSs come out.
Even so, they have not been perceived as fully fledged mathematical concepts but rather as some offending impaired derivatives of the normal, usually hyperbolic IFSs and which must be looked after because of the latter. One can distinguish three general approaches to discrete approximation of hyperbolic IFS invariant sets and measures: the one that utilizes the random iteration algorithm, the one that uses the iterates of the Hutchinson and Markov operators, and the one based on a construction by Williams\cite{Will71}. In this paper we focus on the dyscrete analysis of the first one, which is the most popular method for approximating IFS attractor and measures. 

The random iteration algorithm (proposed by Barnsley and Demko for hyperbolic IFSs in \cite{Barn85} and then extended to IFSs contractive on average with the aid of Elton's ergodic theorem\cite{Elto87}) works, in principle, in the domain of floating point numbers. Therefore, the computation is usually performed with much higher accuracy than the resolution at which a result (an approximation of the attractor or the invariant measure) yielded by the algorithm is stored in memory. Nevertheless, the floating point domain is also a discrete space, so in one way or another we have to do with a discrete IFS that acts here, at best, on a discrete space of floating point numbers. To the best of our knowledge, the first analysis of a computer implementation of the random iteration algorithm in its initial, recreational version, the chaos game, to play with renderings of Sierpi\'nski's triangles, is done, in terms of Markov chains and orbit shadowing, by Goodman in \cite{Good91}. But a milestone was a book by Peruggia\cite{Peru93}, which is entirely dedicated to the analysis of the random iteration algorithm running for discretized versions of hyperbolic IFSs. This book opened our eyes to the problems connected with discretization of IFSs, showing that despite the hyperbolicity of an IFS and regardless of the precision of computation, a discretized version of a hyperbolic IFS may no longer be hyperbolic and may possess multiple attractors and invariant measures. In the context of contemporary publications, especially textbooks in which the hyperbolicity and the resulting uniqueness of attractors and measures were foregrounded, these results were quite surprising. However, there is a problem with treating the book as providing a full theoretical description of the actual, numerical behavior of the random iteration algorithm for hyperbolic IFSs, because the majority of the theorems demonstrated there are constructed on the assumption that one of the IFS maps has its fixed point coinciding with the origin of the Cartesian coordinate system in a discrete space. Obviously, as Peruggia advocates this assumption, every IFS can be appropriately translated to meet this condition, yet there always remains the question about the result of the algorithm if such an imposed transformation is not applied. In addition, the book is focused mainly on the two-dimensional discrete space of pixels and not every theorem demonstrated there easily generalizes to higher dimensions (we will encounter at least one instance in this paper). 
Nevertheless, Peruggia's book is the first and, to our knowledge, sole publication that is entirely focused on the problem of discretization of IFS and takes it up with appropriate mathematical rigor. Its thematic uniqueness makes it an offer you can't refuse whenever a discussion in a book or paper touches upon (occasionally) the subject of IFSs and discretization (see e.g. \cite{Barn05,Barn06,Barn10}).

\section{Organization of the paper}
In Sec.~\ref{sec_minabset} we define fundamental concepts of this paper such as a $\delta$-grid space and a rounding operator, but first of all we generalize the concept of a minimal absorbing set. Originally, the concept was introduced by Peruggia in \cite{Peru93} in connection with discrete versions of contractions on $\Real^2$, as discrete counterparts of fixed points of the latter. We extend this idea with respect to general mappings on a discrete domain and treat discrete versions of contractions as a special case of this generalization. Analogously to a fixed point, which is the attractor of a single contraction, a minimal absorbing set can be viewed as a family of attractors of a (usually non-contractive) map on a discrete space. 
Minimal absorbing sets and their component attractors accompanied with basins of attraction are elementary concepts which we will utilize on numerous occasions in this paper. 

In Sec.~\ref{sect_more_often} we deal with the probability of losing the contractive property by a contraction due to discretization of the Euclidean space (the Euclidean metric). The study is done in terms of the structure of a minimal absorbing set. We examine some relationships between discretization parameters and the structure of minimal absorbing sets that result from the process of discretization of a contraction. Then we give a theoretical foundation for numerical experiments concerning statistics pertaining to the structure of the sets resulting from discretizing two-dimensional affine contractions. We discuss the results of the experiments and show that the probability of losing the contractive property due to discretization by an affine contraction is independent of discretization accuracy and, moreover, it is higher than the probability that a discretized version of the map retains contractivity.  

In Sec.~\ref{Sec_DIFS} we study discrete iterated function systems with place-dependent probabilities (DIFS) as self-standing constructions defined at the very beginning on a discrete space. A discretized version of a regular hyperbolic IFS with constant probabilities comes up in the discussion as a special case of this general construction. We begin the analysis of DIFS invariant sets and measures from the point of view of the random iteration algorithm and Markov chains on countable state spaces (Sec.~\ref{sec_statDist}). We look at the measures from the standpoint of their attractive and ergodic properties and the consequent possibility of approximating the measures (and thus also their supports) by means of the random iteration algorithm in Sec.~\ref{sec_be_attractive}. 

In the meantime, in Sec.~\ref{on_multiattractor}, we review the possibility of the existence of multiple attractors for the case of discretized versions of hyperbolic IFSs. The DIFSs arising from hyperbolic IFSs are studied in the context of the random iteration algorithm in more detail in Sec.~\ref{sec_RIA_hyperbolic}. The theorems offered in this part can be regarded as a generalization of results obtained by Peruggia\cite{Peru93} for the random iteration algorithm running in a pixel space. These theorems generalize the results in that they do not assume a certain specified dimension of space nor do they make some additional assumptions on the relation between a hyperbolic IFS and a discrete space (such as, for instance, a fixed point of an IFS map coinciding with the origin of the coordinate system in a discrete space).      

Finally, in Sec.~\ref{sec_conclude} we remark on some open problems concerning DIFSs themselves as well as their special cases that arise as derivatives of iterated function systems on continuous metric spaces.      

\section{Minimal absorbing sets}\label{sec_minabset}
In this section we develop a theory of minimal absorbing sets for maps acting in a discrete space. In some respects a minimal absorbing set of a map in the discrete space can be viewed as the counterpart of fixed points of contractions in $\Real^n$. However we will define and analyse minimal absorbing sets in isolation from contractive maps in $\Real^n$, as freestanding structures that can arise from the dynamics of a single map in a discrete space. Then we will show that discretization of a contraction in $\Real^n$ leads to a map in a discrete space that can be regarded as a special case of the more general construction we will have considered earlier. We take advantage of the concept of minimal absorbing set throughout this paper, and such a generalization allows us to study, in Sec.~\ref{Sec_DIFS}, discrete iteration function systems on their own, with the ones that arise from discretization of hyperbolic IFSs as a special case.     

Hereafter, we will assume that the metric $d$ on $\Real^n$ is induced by a norm, that is, $d(x, y) := \left \| x - y \right \|$. Such metrics possess a number of properties, which accord with intuition, but are not present in the general case of a metric on $\Real^n$. First, we get that every metric induced by a norm is translation invariant, i.e., for any $x, y, t \in \Real^n$, $d(x + t, y + t) = d(x, y)$. This property, along with the definition of the norm, implies that any ball $B(c, r) = \{x \in \Real^n : d(c, x) < r\}$ is centrally symmetric. In turn, it follows that the diameter $\diam(B(c, r)) = r \diam(B(\mathbf{0}, 1)) = 2r$ and that the distance of the farthest boundary point of the ball from the ball's center is equal to half the ball's diameter, $\sup\{d(c, x) : x \in B(c, r) \} = \diam(B(c, r))/2 = r$. 

The following definition introduces the fundamental structure of this paper, that is, a discrete counterpart of $\Real^n$: 
\newline
\begin{definition}
Let $\Grid^{n}(\delta)$ be the regular tiling of $\mathbb{R}^n$ by disjoint, half-open $n$-dimensional cubes of side $\delta > 0$, which are defined by
\begin{equation}\label{delta_cube}
    C_{\delta}(m_1,\dots, m_n) := \Big[(m_1 - \tfrac{1}{2})\delta, (m_1 + \tfrac{1}{2})\delta \Big)\times \dots \times \Big[(m_n - \tfrac{1}{2})\delta, (m_n + \tfrac{1}{2})\delta \Big), 
\end{equation}
for $m_1,\dots, m_n \in \mathbb{Z}$, and thus 
\begin{equation*}
    \Grid^{n}(\delta) := \big\{C_{\delta}(m_1,\dots, m_n) : m_1,\dots, m_n \in \mathbb{Z} \big\}.
\end{equation*}
We will call $\Grid^{n}(\delta)$ the \emph{$\delta$-grid} in $\mathbb{R}^n$. We define the \emph{$\delta$-discretization} of $\Real^n$, and denote it by $\mathcal{D}^n(\delta)$, as the set of the centers of the cubes in $\Grid^{n}(\delta)$, that is,   
\begin{equation*}
    \mathcal{D}^n(\delta) := \big\{ \delta[m_1, \dots, m_n] : m_1,\dots, m_n \in \mathbb{Z} \big\}.
\end{equation*}
Moreover, for the sake of brevity, we will write $\theta$ to denote $\diam_d(C_{\delta})/2$, half of the diameter of a $\delta$-cube with respect to the metric $d$. 
\end{definition}
\medskip
The discretization space $\mathcal{D}^n(\delta)$ is a formal model of the actual space upon which the concrete implementations of algorithms really act, and from which they yield their results. For example, if we regard squares $C_{\delta}(m_1, m_2)$ as pixels of size $\delta$, then $\mathcal{D}^2(\delta)$ can be identified with an image space. However, when it comes to floating-point arithmetic, the issue is more complicated. Due to a fixed number of significant digits\footnote{Typically the 23-bit and 52-bit mantissa for single- and double-precision, which translates to about 7 and 16 decimal digits.} in the floating point number representation, the representable numbers are not evenly spaced and the distance between consecutive numbers grows with scale\footnote{More strictly, the numbers are evenly spaced in the intervals $[2^j, 2^{j+1}]$, possibly excluding one or both endpoints, and at the endpoints the interval between adjacent numbers doubles\cite{Harr99}.}. In the context of the $\delta$-grid model, this means that to represent the natural setting for floating point arithmetic we would have to use a grid with $\delta$ a monotonically increasing function of real numbers. Nevertheless, to keep things as simple as possible, in such a case we can assume constant $\delta$ set to a certain tiny number that roughly reflects the error carried by the floating point approximation in a bounded interval $[-a, a]$. For example, we can follow the approach used in numerical analysis and accept $\delta$ to be equal to the machine accuracy, which is the smallest (in magnitude) floating point number which, when added to the floating point $1.0$, produces a floating point result different from $1.0$ (see e.g.~\cite{Pres92}). 

In the sequel, for convenience, we will often treat $\mathcal{D}^n(\delta)$ as a subset of $\Real^n$, that is, without defining any explicit mapping of points from $\mathcal{D}^n(\delta)$ to $\Real^n$; in other words, we assume that any point of $\mathcal{D}^n(\delta)$ is by definition a point of $\Real^n$. The converse is in general not true and is the subject of the next definition which establishes a relationship between points and mappings on $\Real^n$ and the ones of $\mathcal{D}^n(\delta)$. 
\medskip
\begin{definition}\label{roundoff_def}
We define the \emph{$\delta$-roundoff of a point} $x \in \Real^n$ as the result of the operation~ $\widetilde{.} : \Real^n \rightarrow \mathcal{D}^n(\delta)$ such that $\tilde{x} = \delta m$, where $m \in \mathbb{Z}^n$ such that $x \in C_{\delta}(m)$. Using the operator, the \emph{$\delta$-roundoff of a mapping} $w : \Real^n \rightarrow \Real^n$ is defined to be the mapping $\tilde{w} : \mathcal{D}^n(\delta) \rightarrow \mathcal{D}^n(\delta)$ such that $\tilde{w}(\tilde{x}) = \widetilde{w(\tilde{x})}$.
\end{definition}
\medskip
In words, given a point in $\Real^n$, the operator $\tilde{.}$ finds a $\delta$-cube within which the point resides and returns the cube's center. Obviously, it is only a conceptual model of actual rounding that goes in real computation environment, in which there are no such things as real numbers to be rounded. The actual process operates all the time only on some representations of ideal reals and the representations are accurate at most for a finite subset of rational numbers. Moreover, as to rounding a mapping $w$, our model is very optimistic because it assumes that the result of a computed and rounded value of $w$ is within $0.5\delta$ of the exact result (w.r.t. the maximum metric $d_{\infty}$), regardless of the class of the mapping itself. This is the best possible accuracy of computation in a given $\delta$-grid setup. However, in floating point computation (IEEE 754 standard) such precision\footnote{Here, we identify the value of $\delta$ with the ULP (acronym for \emph{unit in the last place}), which is a measure of accuracy of floating point arithmetic, defined usually (but not always) as the gap between the two floating-point numbers nearest the number to be rounded\cite{Harr99,Mull05}.} is guaranteed only for elementary arithmetic operations (addition, subtraction, multiplication, division, and square root). Since mappings are typically composed of more than one elementary operation, the error accumulates. Nevertheless, if a mapping consists of a finite number of elementary operations (such as affine mappings, for instance), then the error of rounding is bounded above by a constant, so our model of rounding remains valid up to a multiple of $\delta$ by a constant. 

It is also worth noting that the operator $\tilde{.}$ is a surjection and in fact any mapping on $\mathcal{D}^n(\delta)$ represents uncountable many mappings on $\Real^n$. For now on, if, in a given context, we do not take advantage of some specific properties of the mapping $w : \Real^n \rightarrow \Real^n$ when discussing its $\delta$-roundoff, for brevity we will often refer to $\tilde{w}$ as a mapping on $\mathcal{D}^n(\delta)$. 

The next definition introduces the concept of an absorbing set---the set that has the ability to attract the orbits $\{ \tilde{w}^{\circ i}(\tilde{x})\}_{i}$ and then trap them forever. 
\medskip
\begin{definition}\label{Def_Abs_Set}
Let $\tilde{w} : \mathcal{D}^n(\delta) \rightarrow \mathcal{D}^n(\delta)$, and let $C$ be a nonempty subset of $\mathcal{D}^n(\delta)$ such that $\tilde{w}(C) \subset C$. We will say that a set $\Lambda \subset C$ is an \emph{absorbing set for $\tilde{w}$ in $C$} if
\begin{equation*}
    \forall \tilde{x} \in C, \exists N \in \Natural, \forall i \geq N, \tilde{w}^{\circ i}(\tilde{x}) \in \Lambda. 
\end{equation*}
\end{definition}
\medskip

In the sequel, if $\Lambda$ is an absorbing set for $\tilde{w}$ in $\mathcal{D}^n(\delta)$, we will just write that $\Lambda$ is an absorbing set for $\tilde{w}$, that is, without pointing out $\mathcal{D}^{n}(\delta)$ as a superset of $\Lambda$. Similarly, if it is obvious from the context which mapping $\tilde{w}$ is in question, for brevity we will often omit explicit indication of the mapping and write that $\Lambda$ is an absorbing set in a set. 

It is easy to prove the properties of absorbing sets listed in the following theorem:
\medskip
\begin{thmm}\label{cor_abs_set} 
For $\tilde{w} : \mathcal{D}^{n}(\delta) \rightarrow \mathcal{D}^{n}(\delta)$ and any nonempty set $C \supset \tilde{w}(C)$ the following statements hold:
\newline
\emph{(a)} Every absorbing set is nonempty (by definition).
\medskip
\newline
\emph{(b)} Since $C$ is nonempty, then by definition $C$ is an absorbing set in $C$, so there is always at least one absorbing set in $C$.
\medskip
\newline
\emph{(c)} If $\{\Lambda_i\}_{1\leq i\leq K}$ is a finite family of $K$ absorbing sets in $C$, then also $\bigcap_{i=1}^K \Lambda_i$ is an absorbing set in $C$.
\medskip
\newline
\emph{(d)} If $\Lambda$ is an absorbing set in $C$, then also $\tilde{w}(\Lambda)$ is an absorbing set in $C$.
\medskip
\newline
\emph{(e)} For any absorbing set $\Lambda$ in $C$, there exists $B \subset \Lambda$ such that $\tilde{w}(B) \subset B$ and $B$ is an absorbing set in $C$. 
\medskip
\newline
\emph{(f)} If $\Lambda$ is an absorbing set in $C$, and $C$ is an absorbing set in $D$, then $\Lambda$ is also an absorbing set in $D$.
\end{thmm}

\medskip
In search of the minimal absorbing set, in the theorem below we look at the result of the intersection of all absorbing sets in a given set. 
\medskip
\begin{thmm}\label{MinAbsorb_thm}
Let $\{ \Lambda_{\alpha}\}_{\alpha \in \mathcal{I}}$ be the (possibly uncountable) family of all absorbing sets for $\tilde{w}$ in $C \subset \mathcal{D}^n(\delta)$. Then $\mathcal{M} := \bigcap_{\alpha \in \mathcal{I}}\Lambda_{\alpha}$ is the set of all periodic points of $\tilde{w}$ in $C$, and thus $\mathcal{M} = \tilde{w}(\mathcal{M})$.
\end{thmm}
\begin{proof}
First we show that $\mathcal{M}$ includes all periodic points in $C$.
We have to show that given any periodic point $\tilde{x} \in C$, $\tilde{x}$ belongs to every absorbing set $\Lambda_{\alpha}$ for $\tilde{w}$. On the contrary let us assume that $\tilde{x} \notin \Lambda_{\alpha}$ for some $\alpha \in \mathcal{I}$. Since $\tilde{x}$ is a periodic point, there exists $k \in \Natural$ such that $\tilde{w}^{\circ k}(\tilde{x}) = \tilde{x}$, and thus $\tilde{w}^{\circ (\alpha k)}(\tilde{x}) = \tilde{x}$ for any $\alpha \in \Natural$.  It follows that for any $N \in \Natural$ there is $i \geq N$ such that $\tilde{w}^{\circ i}(\tilde{x}) = \tilde{x} \notin \Lambda_{\alpha}$, which contradicts the assumption of that $\Lambda_{\alpha}$ is an absorbing set for $\tilde{w}$. Therefore, each absorbing set $\Lambda_{\alpha}$ has to include all periodic points in $C$, and hence the periodic points are included in the intersection of the absorbing sets. 

Now we show that if $\tilde{x} \in \mathcal{M}$, then $\tilde{x}$ must be periodic. On the contrary, let us assume that $\tilde{x} \in \mathcal{M}$ and $\tilde{x}$ is non-periodic, which implies that  $\tilde{w}^{\circ i}(\tilde{x}) \neq \tilde{x}$ for all $i\in \Natural$. Let $\Lambda$ be any absorbing set such that $\tilde{x} \in \Lambda$. We will show that $\Lambda\setminus \{\tilde{x}\}$ is an absorbing set in $C$ and thus $\tilde{x}$ cannot be in $\mathcal{M}$. Since $\Lambda$ is absorbing in $C$, we get that for any $\tilde{y} \in C$ there is $N(\tilde{y}) \in \Natural$ so that $\tilde{w}^{\circ i}(\tilde{y}) \in \Lambda$ for all $i \geq N(\tilde{y})$. Therefore, if $\tilde{y}$ is such that $\tilde{w}^{\circ i}(\tilde{y}) \neq \tilde{x}$ for all $i \in \Natural$, then for all $i \geq N(\tilde{y})$, $\tilde{w}^{\circ i}(\tilde{y}) \in \Lambda\setminus \{\tilde{x}\}$. Hence $\Lambda\setminus \{\tilde{x}\}$ absorbs the orbits $\{ \tilde{w}^{\circ i}(\tilde{y})\}_{i\in \Natural}$ that do not intersect $\{ \tilde{x}\}$. Now, because $\tilde{x}$ is non-periodic, $\{ \tilde{w}^{\circ i}(\tilde{x})\}_{i\in \Natural}$ does not intersect $\{\tilde{x} \}$, and thus $\tilde{w}^{\circ i}(\tilde{x}) \in \Lambda\setminus \{\tilde{x}\}$ for all $i\geq N(\tilde{x})$. Therefore, for any $\tilde{y} \in C$ such that $\tilde{w}^{\circ k}(\tilde{y}) = \tilde{x}$ for a certain $k \in \Natural$, we get that $\tilde{w}^{\circ i}(\tilde{y}) \in \Lambda\setminus \{\tilde{x}\}$ for all $i \geq N(\tilde{x})+k$. Hence $\Lambda\setminus \{\tilde{x}\}$ also absorbs the orbits that do intersect $\{ \tilde{x}\}$. Therefore, $\Lambda\setminus \{\tilde{x}\}$ is an absorbing set in $C$, and it follows that $\tilde{x}$ is not in $\mathcal{M}$, so we have arrived at a contradiction, which completes the proof.   
\end{proof}
\medskip
\begin{remark}\label{rem_m_nonabsor}
Note that although $\mathcal{M}$ arises from the intersection of  absorbing sets, it does not follow that $\mathcal{M}$ is an absorbing set itself. In general, when $C \supset \tilde{w}(C)$ is unbounded (and hence the family of the absorbing sets in $C$ may be uncountable), there is no guarantee that the intersection is nonempty, and even if it is not, $\mathcal{M}$ does not have to attract points lying outside. As an example of the latter case consider a $\delta$-roundoff of the mapping $w(x) = \lambda x$, $\lambda > 1$, with $C = \mathcal{D}^n(\delta)$. There are infinitely many (unbounded) absorbing sets, each includes $\mathbf{0}$, and $\mathcal{M} = \{\mathbf{0}\}$, because $\mathbf{0}$ is the only periodic point of $\tilde{w}$. However, $\mathcal{M}$ does not have the attractive property of an absorbing set required by Def.~\ref{Def_Abs_Set}. 
\end{remark}
\medskip
Now observe that for any nonempty set $B \subset \mathcal{D}^n(\delta)$ such that $B \supset \tilde{w}(B)$, an orbit $\{ \tilde{w}^{\circ i}(\tilde{x})\}$ of $\tilde{x}\in B$ does not have to visit all points in $B$. This observation suggests that in general $B$ can be divided into disjoint subsets that are domains for some orbits in $B$ and which are omitted by the remaining orbits in $B$---the ones that stay in the boundaries of the other subsets of the division. To this end, we define a relation $\overset{orb}{\sim}$ on $B$, which given a couple of points in $B$ checks if the orbits of the couple coincide at a certain point:  
\begin{equation}
    \tilde{x}\overset{orb}{\sim} \tilde{y} := \{ (\tilde{x}, \tilde{y}) \in B^2 : \exists j, k \in \Natural\; s.t. \; \tilde{w}^{\circ j}(\tilde{x}) = \tilde{w}^{\circ k}(\tilde{y})\; \}.  
\end{equation}
It is easy to see that the relation is an equivalence relation and thus uniquely decomposes $B$ into the union of disjoint nonempty subsets, $B = \bigcup_i B_i$ with $B_i$ being the equivalence classes of the relation. By definition, each $B_i$ is composed of  orbits in $B$ that meet at a certain point and then, naturally, coincide from that point on. Furthermore, since $\mathcal{D}^n(\delta)$ is countable and $B_i$ are disjoint, the decomposition consists of a countable number of sets $B_i$. Moreover, if $B$ is additionally an absorbing set in a set $C$, then the decomposition of $B$ extends to $C$, resulting in a countable partition of $C$ into the sets of the form $\{\tilde{x} \in C : \exists i\in \Natural \; s.t. \; \tilde{w}^{\circ i}(\tilde{x}) \in B_i \}$. 

We summarize the above observations in the form of the following definition concerning the minimum of absorbing sets, their components and basins of attraction:
\medskip
\begin{definition}\label{minabsorbing_def}
If the unique set $\mathcal{M} \subset C$ defined in Theorem~\ref{MinAbsorb_thm} meets the conditions of an absorbing set in $C$ we will refer to it as the \emph{minimal absorbing set for $\tilde{w}$ in $C$} and denote it by $\mathcal{M}[\tilde{w}, C]$. If $\mathcal{M}$ is the minimal absorbing set for $\tilde{w}$ in the whole space $\mathcal{D}^n(\delta)$, we will say that $\mathcal{M}$ is just the \emph{minimal absorbing set for $\tilde{w}$} and denote it by $\mathcal{M}[\tilde{w}]$. The disjoint, nonempty sets of the unique, countable decomposition of $\mathcal{M}[\tilde{w}, C]$ with respect to the relation $\overset{orb}{\sim}$ will be called the set's \emph{components} and denoted by $\mathcal{M}_i[\tilde{w}, C]$, $i = 1, \dots$. Finally, the \emph{basin of attraction} of a component $\mathcal{M}_i[\tilde{w}, C]$ is defined by
\begin{equation*}
    \mathcal{B}[\mathcal{M}_i[\tilde{w}, C]] := \{ \tilde{x} \in C : \exists i \in \Natural \textit{ s.t. } \tilde{w}^{\circ i}(\tilde{x}) \in \mathcal{M}_i[\tilde{w}, C] \}. 
\end{equation*}
\end{definition}

\medskip
Clearly, periodic points $\tilde{x}$ and $\tilde{y}$ are in the coincidence orbit relation, $\tilde{x} \overset{orb}{\sim} \tilde{y}$, if and only if they share the same periodic orbit. Since, by Theorem~\ref{MinAbsorb_thm}, the minimum set $\mathcal{M}[\tilde{w}, C]$ consists of periodic points, we immediately get the following:
\medskip
\begin{corollary}\label{component_periodic}
Each component $\mathcal{M}_i[\tilde{w}, C]$ of the minimal absorbing set is equal to one of the periodic orbits for $\tilde{w}$ in $C$. Hence, the cardinality of every component is finite and
\begin{equation*}
    \tilde{w}(\mathcal{M}_i[\tilde{w}, C]) = \mathcal{M}_i[\tilde{w}, C].
\end{equation*}
\end{corollary}
\medskip

As shown in Remark \ref{rem_m_nonabsor}, an intersection of absorbing sets does not always yield an absorbing set. The following theorem gives a sufficient condition for the existence of the minimal absorbing set:
\medskip
\begin{thmm}\label{bounded_min}
If the set $C \subset \mathcal{D}^n(\delta)$ is bounded and $C \supset \tilde{w}(C)$, then the minimal absorbing set  $\mathcal{M}[\tilde{w}, C]$ exists.
\end{thmm}
\begin{proof}
By Theorem \ref{cor_abs_set} (b) there is at least one absorbing set in $C$. Moreover, $C$ is bounded and thus finite, so there is at most a finite number of absorbing sets in $C$. Hence, the the conclusion follows from Theorem \ref{cor_abs_set} (c). 
\end{proof}
\medskip
The next theorem shows that the inheritance of absorbing sets by extensions of the original domain, asserted in Theorem \ref{cor_abs_set} (f), also holds for minimum of absorbing sets with preserving the property of minimality.  
\medskip
\begin{thmm}\label{inherit_min}
If $\mathcal{M}$ is the minimal absorbing set for $\tilde{w}$ in $C$, and $C$ is an absorbing set for $\tilde{w}$ in $D$, then $\mathcal{M}$ is also the minimal absorbing set for $\tilde{w}$ in $D$, that is,
\begin{equation*}
    \mathcal{M}[\tilde{w}, C] = \mathcal{M}[\tilde{w}, D]. 
\end{equation*}
\begin{proof}
By Theorem \ref{MinAbsorb_thm}, $\mathcal{M}[\tilde{w}, C]$ consists of all periodic points of $\tilde{w}$ in $C$. Naturally, the points remain periodic in every superset of $C$, in this instance the set $D$. Hence, every absorbing set in $D$ has to include all periodic points in $D$ (otherwise the set would not be absorbing in $D$), and we get that the intersection of all absorbing sets in $D$ is nonempty and include $\mathcal{M}[\tilde{w}, C]$. Moreover, $\mathcal{M}[\tilde{w}, C]$ is an absorbing set in $C$, and by the assumption of the theorem, $C$ is an absorbing set in $D$, so from Theorem \ref{cor_abs_set} (f) it follows that $\mathcal{M}[\tilde{w}, C]$ is also an absorbing set in $D$, and thus $\mathcal{M}[\tilde{w}, C]$ include the intersection of all absorbing sets in $D$. Since, as we have shown, $\mathcal{M}[\tilde{w}, C]$ is also included in the intersection, this completes the proof.  
\end{proof}
\end{thmm}
Now we will take a closer look at the existence of minimal absorbing sets for $\delta$-roundoffs of contractions. We begin with the following single-map-orbit shadowing lemma, which is a special case of Lemma \ref{DIFS_shadowing_thm} proved in Sec.~\ref{Sec_DIFS}:
\medskip
\begin{lem}
Let $\tilde{w} : \mathcal{D}^n(\delta) \rightarrow \mathcal{D}^n(\delta)$ be the $\delta$-roundoff of a contraction $w : \Real^n \rightarrow \Real^n$ with respect to the metric $d$. Let $\{\tilde{w}^{\circ i}(\tilde{x})\}_{i=1}^{\infty}$ and $\{w^{\circ i}(\tilde{x})\}_{i=1}^{\infty}$ be orbits of an arbitrary point $\tilde{x} \in \mathcal{D}^n(\delta)$. Then
\begin{equation}\label{ineqconv2}
    d(\tilde{w}^{\circ i}(\tilde{x}), w^{\circ i}(\tilde{x})) \leq \theta(1-\lambda)^{-1}, \; \forall i \in \mathbb{N}, 
\end{equation}
where $\lambda \in [0, 1)$ is the contractivity factor of $w$.
\end{lem}
\medskip
The upper bound of the form $\theta (1-\lambda)^{-1}$ is symptomatic of $\delta$-discretization, and we will encounter it many times in this paper. On the basis of Eq.~\eqref{ineqconv2} we see that the orbit $\{\tilde{w}^{\circ i}(\tilde{x}))\}$ is shadowed by the orbit $\{w^{\circ i}(\tilde{x})\}$, so we expect that, in the limit, the fixed point $x_f = w(x_f)$ will be imitated by a bounded orbit in $\mathcal{D}^n(\delta)$.  
\medskip
\begin{thmm} \label{thm_alpha_eps}
Let $\tilde{w} : \mathcal{D}^n(\delta) \rightarrow \mathcal{D}^n(\delta)$ be the $\delta$-roundoff of a contraction $w : \Real^n \rightarrow \Real^n$ with respect to the metric $d$. Let $x_f \in \Real^n$ be the fixed point of $w$, and $\lambda \in [0, 1)$ be the map's contractivity factor. Let $\Lambda(x_f, r) := \overline{B}(x_f, r) \cap \mathcal{D}^n(\delta)$, where $\overline{B}(x_f, r)$ is the closed ball in $(\Real^n, d)$, centred at $x_f$ and with radius $r$. Then $\Lambda(x_f, r_0)$ with $r_0 = \theta(1-\lambda)^{-1}$, that is,
\begin{equation}\label{alpha_eps}
    \Lambda(x_f, r_0) = \{ \tilde{y} \in \mathcal{D}^n(\delta) : d(\tilde{y}, x_f) \leq \theta(1-\lambda)^{-1}\}, 
\end{equation}
is an absorbing set for $\tilde{w}$. Moreover, $\tilde{w}$ maps $\Lambda(x_f, r_0)$ into itself,
\begin{equation}
    \tilde{w}(\Lambda(x_f, r_0)) \subset \Lambda(x_f, r_0). 
\end{equation}
\end{thmm}
\begin{proof}
Let $\tilde{x}$ be any point of $\mathcal{D}^n(\delta)$. The mapping $w$ is a contraction on $(\Real^n, d)$, so $\lim_{i \rightarrow \infty} w^{\circ i}(\tilde{w}) = x_f$ or equivalently
\begin{equation*}
    \forall \varepsilon > 0, \exists N \in \Natural, \forall i \geq N, d(w^{\circ i}(\tilde{x}), x_f) < \varepsilon.
\end{equation*}
As a consequence, for a given $\varepsilon > 0$, there exists $N \in \Natural$ such that for any $i \geq N$ we get that
\begin{equation*}
\begin{split}
    d(\tilde{w}^{\circ i}(\tilde{x}), x_f) \leq d(\tilde{w}^{\circ i}(\tilde{x}), w^{\circ i}(\tilde{x})) + d(w^{\circ i}(\tilde{x}), x_f) < \theta(1-\lambda)^{-1} + \varepsilon
\end{split}
\end{equation*}
on the basis of inequality \eqref{ineqconv2}. Hence, for all $i \geq N$, $\tilde{w}^{\circ i}(\tilde{x}) \in \Lambda(x_f, r_0 +\varepsilon)$, and because $\tilde{x}$ is any point of $\mathcal{D}^n(\delta)$, we get that, for every $\varepsilon > 0$, $\Lambda(x_f, r_0+\varepsilon)$ is an absorbing set for $\tilde{w}$.

Now we prove that $\Lambda(x_f, r_0 + \varepsilon)$ also owns the absorbing property for $\varepsilon = 0$. Let $\Lambda^C(x_f, r_0) = \mathcal{D}^n(\delta) \setminus  \Lambda(x_f, r_0)$. Since $\Lambda^C(x_f, r_0)$ is countable, the minimum $\varepsilon = \min \{ d(\tilde{y}, x_f) : \tilde{y} \in \Lambda^C(x_f, r_0) \}$ exists and $\varepsilon > r_0$. Therefore, $\Lambda(x_f, r_0 + \varepsilon/2)$ does not include any point from $\Lambda^C(x_f, r_0)$, and thus $\Lambda(x_f, r_0 + \varepsilon/2) = \Lambda(x_f, r_0)$. 

The last thing to we show is that $\tilde{w}$ maps $\Lambda(x_f, r_0)$ into itself. Let $\tilde{y} \in \Lambda(x_f, r_0)$. We have 
\begin{equation*}
\begin{split}
    d(\tilde{w}(\tilde{y}), x_f) \leq d(\tilde{w}(\tilde{y}), w(\tilde{y})) + 
 d(w(\tilde{y}), x_f) 
    \leq \theta + \lambda d(\tilde{y}, x_f) 
    \leq \theta \big(1 + \lambda(1-\lambda)^{-1}\big)
    = \theta (1-\lambda)^{-1} 
\end{split}
\end{equation*}
because $d(\tilde{w}(\tilde{y}), w(\tilde{y})) \leq \theta$ by the definition of a $\delta$-roundoff of a mapping (Def.~\ref{roundoff_def}). Therefore, $\tilde{w}(\tilde{y}) \in \Lambda(x_f, r_0)$, which completes the proof.  
\end{proof}

\begin{corollary}\label{contraction_minimumset}
If $\tilde{w} : \mathcal{D}^n(\delta)\rightarrow \mathcal{D}^n(\delta)$ is the roundoff of a contraction $w : \Real^n \rightarrow \Real^n$, then the minimal absorbing set  $\mathcal{M}[\tilde{w}]$ exists. Morover, $\mathcal{M}[\tilde{w}]$ is finite and its cardinality is bounded from above by the cardinality of $\Lambda(x_f, r_0)$.
\end{corollary}
\begin{proof}
On the basis of Theorem \ref{thm_alpha_eps}, $\Lambda(x_f, r_0)$ is bounded and $\tilde{w}$ maps it into itself, so from Theorem \ref{bounded_min} there exists $\mathcal{M}[\tilde{w}, \Lambda(x_f, r_0)]$. But, by Theorem \ref{thm_alpha_eps}, $\Lambda(x_f, r_0)$ is an absorbing set for $\tilde{w}$ in $\mathcal{D}^n(\delta)$, and thus $\mathcal{M}[\tilde{w}]=\mathcal{M}[\tilde{w}, \Lambda(x_f, r_0)]$ by Theorem \ref{inherit_min}. \end{proof}

\begin{figure}[t]
  \centering
  \includegraphics[width=0.8\textwidth]{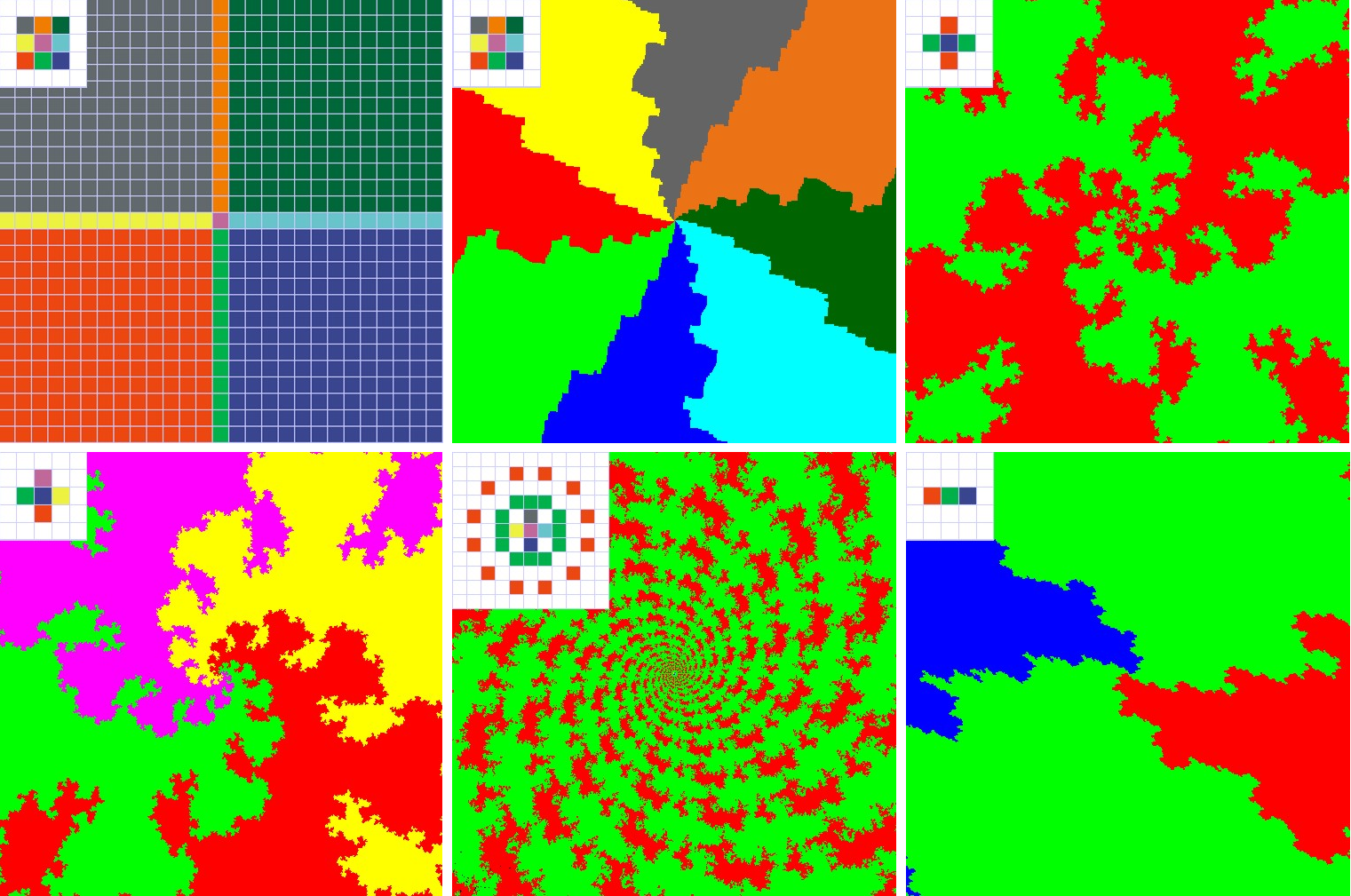}
  \caption{Examples of minimal absorbing sets and their basins of attractions generated by $\delta$-roundoffs of two-dimensional linear contractions. From left to right, top to bottom, the first five pictures are derived from similarities with scaling factors and rotations respectively: $0.6$ and $0^{\circ}$, $0.6$ and $5^{\circ}$,  $0.6$ and $150^{\circ}$, $0.6$ and $30^{\circ}$, $0.9$ and $30^{\circ}$. The last picture is derived from a linear mapping specified by the matrix $\big[\begin{smallmatrix}
  0.5 & 0.3\\
  -0.1 & 0.4
\end{smallmatrix}\big]$. }
  \label{fig:1}
\end{figure}

In Fig.~\ref{fig:1} we present some examples of minimal absorbing sets and their basins of attractions generated by $\delta$-roundoffs of two-dimensional affine contractions.

\section{...more often than rarely}\label{sect_more_often}
In this section we examine the probability that discretization of a contraction mapping $w$ in $(\Real^n, d)$ results in a mapping $\tilde{w}$ that is deprived of the contractive property in $(\mathcal{D}^n(\delta), d)$. We base our study on the observation that a necessary (but not sufficient) condition for a mapping $\tilde{w}$ to be a contraction in $(\mathcal{D}^n(\delta), d)$ is that the minimal absorbing set for $\tilde{w}$ has to consist of only one point. On that basis, the probability that discretization retains the contractivity of a mapping is bounded from above by the probability for the resulting $\delta$-roundoff to have a minimal absorbing set being a singleton. Since the probability in question is related to the cardinality of minimal absorbing sets, it is critical how the cardinality relates to the value of the parameter $\delta >0$ that controls the accuracy of discretization. In the sequel we will show that in the case of affine mappings in $(\Real^n, d_E)$, $d_E$ the Euclidean metric, the probability $p$ that a $\delta$-roundoff of an affine contraction has a non-singleton minimal absorbing set is positive and remains constant for all values of $\delta$. Therefore, given an affine contraction, the resulting minimal absorbing set for a certain $\delta_1$ may be as much a singleton as a non-singleton with probabilities $1-p$ and $p$, respectively, and independently of any other value $\delta_2$ for which the minimal absorbing set is a singleton or a non-singleton with the same probabilities. As a result, not only does an increase of the discretization precision not make the minimal absorbing set head for a singleton, but also it may shift a singleton minimal absorbing set to a non-singleton one. At the end of this section we present results of some numerical experiments conducted for two-dimensional affine contractions, among others we give numerical evidence that in this case the probability $p$ for the occurrence of a non-singleton minimal absorbing set is higher than $1-p$, the probability of the occurrence of a singleton one.   

We begin with a theorem concerning a sufficient condition for a general contraction (i.e., not necessarily being affine) in $(\Real^n, d_E)$ to have a singleton minimal absorbing set in a special case in which the fixed-point of a contraction coincides with a point in $\mathcal{D}^n(\delta)$ (i.e., the center of a $\delta$-cube). In \cite{Peru93} Peruggia demonstrated a simple proof for the case of contractions in $(\Real^2, d_E)$ (cf. Lemma 4.26, pp.~84--85). The generalization of this result to $n$-dimensions requires more effort. First we need the following two lemmas:    
\medskip
\begin{lem}\label{lem_6}
Let $\tilde{x}, \tilde{o} \in \mathcal{D}^n(\delta)$ such that $d_E(\tilde{x}, \tilde{o}) < \delta \sqrt{k}$, $0 < k \leq n$, where $d_E$ denotes the Euclidean metric on $\Real^n$. Then $\tilde{x}$ belongs to at least one of the affine subspaces $\tilde{o} + V^{(k-1)} = \{ \tilde{o} + v : v \in V^{(k-1)}\}$, where $V^{(k-1)}$ is a linear subspace spanned by $k-1$ different vectors from the standard basis of $\Real^n$ (i.e., each basis vector has a single nonzero entry with value $1$). Moreover, $d_E(\tilde{x}, \tilde{o}) \leq \delta \sqrt{k-1}$.
\end{lem}
\begin{proof}
Let $\tilde{x}, \tilde{o} \in \mathcal{D}^n(\delta)$ and $d_E(\tilde{x}, \tilde{o}) < \delta \sqrt{k}$, $k \leq n$. Then, because $\delta >0$, 
\begin{equation*}
    \delta\sqrt{k} > \|\tilde{x} - \tilde{o} \|_E = \delta\|\tilde{x}/\delta - \tilde{o}/\delta \|_E 
\end{equation*}
and hence
\begin{equation*}
\|\tilde{x}/\delta - \tilde{o}/\delta \|_E^2 < k.   
\end{equation*}
But $\tilde{x}/\delta - \tilde{o}/\delta \in \mathbb{Z}^n$, so $\|\tilde{x}/\delta - \tilde{o}/\delta \|_E^2 \in \Natural_0$, and therefore the above inequality implies that 
\begin{equation}\label{sqnorm}
    \|\tilde{x}/\delta - \tilde{o}/\delta \|_E^2 \leq k-1.
\end{equation}
It follows that $\|\tilde{x}/\delta - \tilde{o}/\delta \|_E \leq \sqrt{k-1}$ and thus
\begin{equation*}
    \delta\sqrt{k-1} \geq \|\tilde{x} - \tilde{o} \|_E = d_E(\tilde{x}, \tilde{o}), 
\end{equation*}
so we have proved the second part of the theorem. Now, from inequality \eqref{sqnorm} and $\tilde{x}/\delta - \tilde{o}/\delta \in \mathbb{Z}^n$ we conclude that the vector $\tilde{x}/\delta - \tilde{o}/\delta$ has at most $k-1$ nonzero entries, and so does the vector $\tilde{x} - \tilde{o} \in \mathcal{D}^n(\delta)$. Therefore, $\tilde{x} - \tilde{o}$ belongs to at least one of the linear subspaces $V^{(k-1)}$ of $\Real^n$ defined in the theorem. From this we immediately get that $\tilde{x}$ belongs to at least one of the affine subspaces $\tilde{o} + V^{(k-1)}$.
\end{proof}
\medskip
\begin{lem}\label{lem_7}
Let $V^{(k)}$ be a linear subspace of $\Real^n$, spanned by $0 < k < n$ different vectors from the standard basis of $\Real^n$. Let $\tilde{x}, \tilde{o} \in \mathcal{D}^n(\delta)$ such that $\tilde{x} \in \tilde{o} + V^{(k)}$. Then for any $x \in C_{\delta}(\tilde{x}/\delta)$, where $C_{\delta}(\tilde{x}/\delta)$ is a $\delta$-cube in $\mathcal{G}^n(\delta)$, we have
\begin{equation*}
    d_E(x, \tilde{o}) \geq d_E(\tilde{x}, \tilde{o}) - \delta \frac{\sqrt{k}}{2}.
\end{equation*}
\end{lem}
\begin{proof}
Let $\tilde{x}, \tilde{o} \in \mathcal{D}^n(\delta)$ and let $C_\delta(\tilde{x}/\delta)$ be a $\delta$-cube in $\mathcal{G}^n(\delta)$ such that $\tilde{x} \in \tilde{o} + V^{(k)}$. Denote $C^{(k)}(\tilde{x}/\delta) := C_\delta(\tilde{x}/\delta) \cap (\tilde{o} + V^{(k)})$. By definition, $C_\delta(\tilde{x}/\delta)$ consists of points $x = \tilde{x} + \delta \varepsilon$, where $\varepsilon = [\varepsilon_1, \dots, \varepsilon_n]$, $\varepsilon_i \in [-1/2, 1/2)$, or equivalently $\delta \varepsilon \in C_\delta(\mathbf{0})$, where $\mathbf{0}$ is the zero vector of $\Real^n$. Now, by our assumption $\tilde{x} \in \tilde{o} + V^{(k)}$ and by definition $C^{(k)}(\tilde{x}/\delta) \subset \tilde{o} + V^{(k)}$, so $C^{(k)}(\tilde{x}/\delta)$ consists of points $x' = \tilde{x} + \delta \varepsilon'$, where $\delta \varepsilon' \in V^{(k)}$ and, moreover, $\delta \varepsilon' \in C_\delta(\mathbf{0})$, because $C^{(k)}(\tilde{x}/\delta) \subset C_\delta(\tilde{x}/\delta)$. This yields $\varepsilon' \in C_1(\mathbf{0})$ and, at the same time, $\varepsilon' \in V^{(k)}$, because $V^{(k)}$ is a linear space. Therefore, $\varepsilon' \in C_1(\mathbf{0}) \cup V^{(k)}$. Since $V^{(k)}$ is spanned by $k$ vectors from the standard basis of $\Real^n$, we conclude that the Euclidean norm $\|\varepsilon'\|_E \leq \sqrt{k (\tfrac{1}{2})^2} = \sqrt{k}/2$. It follows that $d_E(\tilde{x}, x') \leq \delta \sqrt{k}/2$, and by the triangle inequality we get that for any $x' \in C^{(k)}(\tilde{x}/\delta)$ and any $y \in \Real^n$
\begin{equation*}
    d_E(x', y) \geq d_E(\tilde{x}, y) - \delta\sqrt{k}/2.
\end{equation*}
Now, for any $x \in C_\delta(\tilde{x}/\delta)$ there exists $x' \in C^{(k)}(\tilde{x}/\delta)$ which is the orthogonal projection of $x$ onto $\tilde{o} + V^{(k)}$. Therefore, for any $y \in \tilde{o} + V^{(k)}$, from the Pythagorean theorem we immediately get that $d_E(x, y) \geq d_E(x', y)$. Since $\tilde{o} \in \tilde{o} + V^{(k)}$, this completes the proof.
\end{proof}
\medskip
\begin{thmm}\label{contr_half_thm}
Let $\mathcal{G}^n(\delta)$ be the $\delta$-grid on $\Real^n$, and let $w: \Real^n \rightarrow \Real^n$ be a contraction with respect to the Euclidean metric $d_E$, with the contractivity factor $\lambda < \tfrac{1}{2}$ and the fixed point $x_f = \tilde{o} \in \mathcal{D}^n(\delta)$. Then the minimal absorbing set $\mathcal{M}[\tilde{w}] = \{\tilde{o}\}$. 
\end{thmm}
\begin{proof}
Since the contractivity factor $\lambda < \tfrac{1}{2}$, on the basis of Theorem \ref{thm_alpha_eps} we know that the points of the minimum set $\mathcal{M}[\tilde{w}] \subset \mathcal{D}^n(\delta)$ satisfy
\begin{equation*}
    \forall \tilde{x} \in \mathcal{M}[\tilde{w}],\; d_E(\tilde{x}, \tilde{o}) < \delta \sqrt{n}.
\end{equation*}
Hence, from Lemma \ref{lem_6} we immediately get that
\begin{equation}\label{eq_10}
    \forall \tilde{x} \in \mathcal{M}[\tilde{w}],\; d_E(\tilde{x}, \tilde{o}) \leq \delta \sqrt{n-1}
\end{equation}
and, moreover, each $\tilde{x} \in \mathcal{M}(\tilde{w})$ belongs to an affine subspace $\tilde{o}+V^{(n-1)}$ with the linear space $V^{(n-1)}$ defined in the lemma. 

In addition, because $\tilde{w}(\mathcal{M}[\tilde{w}]) = \mathcal{M}[\tilde{w}]$ (Theorem \ref{MinAbsorb_thm}), for any $\tilde{x} \in \mathcal{M}[\tilde{w}]$ there exists some $\tilde{y} \in \mathcal{M}[\tilde{w}]$ such that $\tilde{w}(\tilde{y}) = \tilde{x}$ or equivalently $w(\tilde{y}) \in C_\delta(\tilde{x}/\delta)$. Therefore, by Lemma \ref{lem_7} we get that
\begin{equation}\label{eq_11}
    \forall \tilde{x} \in \mathcal{M}[\tilde{w}],\; d_E(\tilde{x}, \tilde{o}) \leq d_E(w(\tilde{y}), \tilde{o}) + \delta\frac{\sqrt{n-1}}{2}. 
\end{equation}
Moreover, since $w$ is a contraction with the contractivity factor $\lambda < 1/2$ and with the fixed point $\tilde{o}$, we have
\begin{equation*}
    d_E(w(\tilde{y}), \tilde{o}) \leq s d_E(\tilde{y}, \tilde{o}) < \tfrac{1}{2} d_E(\tilde{y}, \tilde{o})
\end{equation*}
and because $\tilde{y} \in \mathcal{M}[\tilde{w}]$, from inequality \eqref{eq_10} we conclude that
\begin{equation*}
    d_E(w(\tilde{y}), \tilde{o}) < \delta\frac{\sqrt{n-1}}{2}. 
\end{equation*}
Using the above inequality in inequality \eqref{eq_11}, as a result we get that
\begin{equation*}
    \forall \tilde{x} \in \mathcal{M}[\tilde{w}],\; d_E(\tilde{x}, \tilde{o}) < \delta \sqrt{n-1}.
\end{equation*}

In turn, the application of Lemma \ref{lem_6} to the above inequality yields $d_E(\tilde{x}, \tilde{o}) \leq \delta \sqrt{n-2}$, and then, by Lemma \ref{lem_7} and the same argument as above, we conclude that $d_E(\tilde{x}, \tilde{o}) < \delta \sqrt{n-2}$ for all $\tilde{x} \in \mathcal{M}[\tilde{w}]$. 

Continuing in this way, by successively decrementing the value in the square root by $1$ in each step, we eventually achieve $d_E(\tilde{x}, \tilde{o}) < \delta$, and then by Lemma \ref{lem_6} $d_E(\tilde{x}, \tilde{o}) = 0$ for all $\tilde{x} \in \mathcal{M}[\tilde{w}]$ as required. 
\end{proof}

\begin{figure}[t]
  \centering
  \includegraphics[width=0.7\textwidth]{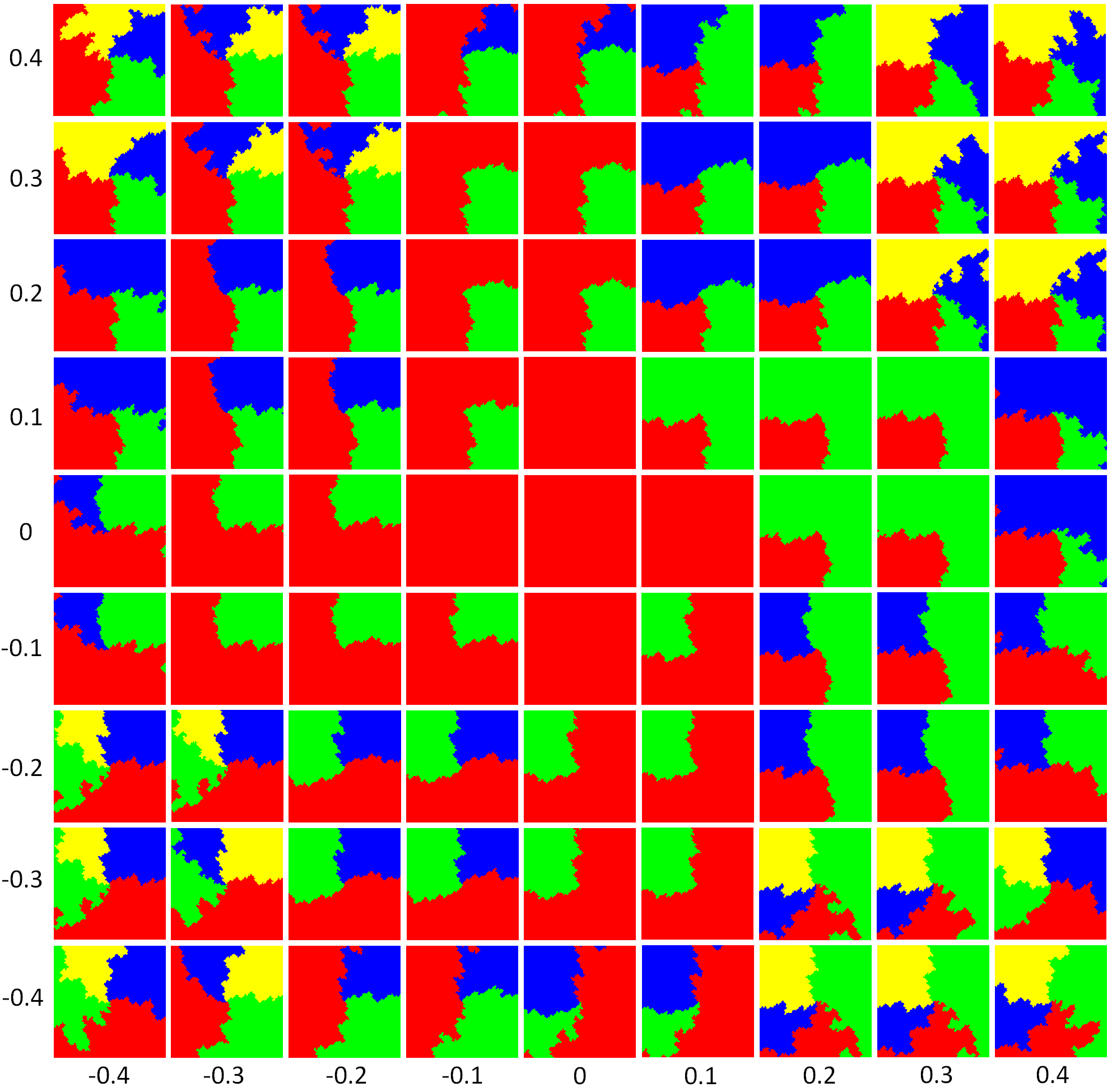}
  \caption{The basins of attractions generated by $\delta$-roundoffs of contractive similarities with the linear part specified by scaling factor $0.45$ and rotation $15^{\circ}$, and fixed points in $[-0.4, 0.4]^2$}
  \label{fig:2}
\end{figure}

The conclusion of the theorem holds, however, only for contractions with fixed-points coincident with centers of $\delta$-cubes (i.e., the points in $\mathcal{D}^n(\delta)$) and, except for the mappings with fixed-points coinciding with the zero vector $\mathbf{0}$, this match is naturally destroyed by a shift of the parameter $\delta$. What is more, due to the so-called "curse of dimensionality", the probability that the fixed point of a randomly chosen contraction is closer to the cube's boundary rather than to its center grows with the dimension $n$ of the space. In effect, even contractions with contractivity factors less than $\tfrac{1}{2}$ may induce minimal absorbing sets that are not singletons. Fig.~\ref{fig:2} offers a glimpse into the relationship between a minimal absorbing sets and the position of a fixed point within a $\delta$-cube. (We revisit this issue at the end of this section while discussing numerical experiments.) Of course, given fixed $\delta$, one can always transform a contraction so as to match its fixed point with any point in $\mathcal{D}^n(\delta)$. However, such a transformation usually alters the dynamics generated by iteration of the map, limiting it to a rather narrow subset of actual behaviors. Besides, in concrete implementations of the algorithms for IFS attractor or measure approximation such an initial transformation, although possible, is hardly ever done in practice, and yet the implementations seem to work properly. And in light of the stated loss of contractivity due to discretization, this is the reason for our study on \emph{why} the algorithms work correctly and, first of all, \emph{if} they really work correctly. Therefore, throughout this paper we do not impose such a limitation on mappings, and allow their fixed points to be located anywhere in $\Real^n$. 

Let us start to investigate the influence of changing the discretization parameter $\delta$ on the minimal absorbing set of a contraction. Shifting the value of $\delta$ is equivalent to scaling of coordinates in $\Real^n$ by the reciprocal of $\delta$ and therefore it alters the corresponding $\delta$-roundoffs of  underlying contractions. In the following, we will examine the impact that scaling in $\Real^n$ has on $\delta$-roundoffs of contractions and thus on the corresponding minimal absorbing sets. In addition, we will also take a look at the effect of applying translation in $\Real^n$ as we will need it in further theorems that serve as the theoretical foundation of numerical experiments whose results are presented at the end of this section. Given $w : \Real^n \rightarrow \Real^n$, let us denote:
\begin{equation}\label{map_trans}
    (w + t)(x) := w(x - t) + t, 
\end{equation}
\begin{equation}\label{map_scale}
    (\alpha w)(x) := \alpha w(x/\alpha),
\end{equation}
where $\alpha \in \Real$ and $t \in \Real^n$. In words, $(w + t)(.)$ and, respectively, $(\alpha w)(.)$ can be viewed as the mappings that result from transformation of the mapping $w(.)$ by translation $t$ and uniform scaling $\alpha$, respectively. 

Translation of a mapping in $\Real^n$ changes the mapping's $\delta$-roundoff, and thus, if the mapping is a contraction, translation may affect properties of the resulting minimal absorbing set, such as its geometry, the number of components, basins of attraction, and its cardinality. However, if a contraction is shifted so that the translation does not change the relative position of the fixed point with respect to the center of a $\delta$-cube in which the point resides, then the only effect the transformation has on the minimal absorbing set is shifting the set by the same vector and without altering any other properties of the set. The following theorem and corollary formalizes this translation invariance of the minimal absorbing set properties in a somewhat more general manner that does not appeal to a fixed point.   
\medskip
\begin{thmm}\label{thm_trans}
Let $w : \Real^n \rightarrow \Real^n$ and $\tilde{w} : \mathcal{D}^n(\delta) \rightarrow \mathcal{D}^n(\delta)$. Then, for any $m\in \mathbb{Z}^n$ and any $x \in \Real^n$,  the $\delta$-roundoffs of $w$ and $(w + \delta m)$ satisfy
\begin{equation}
    (\widetilde{w + \delta m})(\tilde{x} + \delta m) = \tilde{w}(\tilde{x}) + \delta m.  
\end{equation}
\end{thmm}
\begin{proof}
Let $\delta > 0$. By Eq.~\eqref{map_trans} we have, for all $x \in \Real^n$ and $m \in \mathbb{Z}^n$,
\begin{equation}\label{eq_trans}
    (w + \delta m)(x + \delta m) = w(x) + \delta m.
\end{equation}
Let $y \in \Real^n$. Then there exists $\varepsilon = [\varepsilon_1, \dots, \varepsilon_n]$, $\varepsilon_i \in [-1/2, 1/2)$, such that $y = \tilde{y} + \delta \varepsilon$. It follows that, for any $m \in \mathbb{Z}^n$,
\begin{equation*}
    \widetilde{y + \delta m} = \widetilde{(\tilde{y} + \delta \varepsilon + \delta m)} = \widetilde{(\tilde{y} + \delta m)} = \tilde{y} + \delta m,
\end{equation*}
because by Eq.~\eqref{delta_cube} and Def.~\ref{roundoff_def}, for any $\tilde{x} \in \mathcal{D}^n(\delta)$, $\widetilde{\tilde{x} + \delta \varepsilon} = \tilde{x}$ , and $\tilde{y} + \delta m \in \mathcal{D}^n(\delta)$. We get the conclusion of the theorem by taking $y = w(\tilde{x})$ in the above equation and applying Eq.~\eqref{eq_trans} to its left-hand side.
\end{proof}
\medskip
Hence, by the definition of a minimal absorbing set (Def.~\ref{minabsorbing_def}), we get the following:
\medskip
\begin{corollary}
If $w : \Real^n \rightarrow \Real^n$ is a contraction and $\tilde{w} : \mathcal{D}^n(\delta) \rightarrow \mathcal{D}^n(\delta)$, then for any $m \in \mathbb{Z}^n$, 
\begin{equation}\label{trans_minimum}
    \mathcal{M}[\widetilde{w + \delta m}] = \mathcal{M}[\tilde{w}] + \delta m,
\end{equation}
where the right-hand side of the equation denotes the translation of each point of $\mathcal{M}[\tilde{w}]$ by vector $\delta m$.  
\end{corollary}
\medskip
A similar simple fact concerning scaling is expressed in the next theorem and corollary: A simultaneous scaling of a contraction on $\Real^n$ and a discrete space $\mathcal{D}^n(\delta)$ by a factor $\alpha > 0$ implies scaling, by $\alpha$, the minimal absorbing set for the $\delta$-roundoff of the contraction. That is, the result is the minimal absorbing set for the $(\alpha\delta)$-roundoff of $\alpha w$, which is an $\alpha$-scaled copy of the minimal absorbing set for the $\delta$-roundoff of $w$.     
\medskip
\begin{thmm}
Let $w : \Real^n \rightarrow \Real^n$. Let $\tilde{w} : \mathcal{D}^n(\delta_1) \rightarrow \mathcal{D}^n(\delta_1)$ and $\widetilde{(\alpha w)} : \mathcal{D}^n(\delta_2) \rightarrow \mathcal{D}^n(\delta_2)$, where $\alpha = \delta_2 / \delta_1$. Then, for any $x \in \Real^n$, 
\begin{equation}
    \widetilde{(\alpha w)}(\widetilde{\alpha x})_{\delta_2} = \alpha \tilde{w}(\tilde{x})_{\delta_1},
\end{equation}
where the indices indicate the $\delta$ coefficients of the roundoffs. 
\end{thmm}
\begin{proof}
Let $y \in \Real^n$ and $\tilde{y} \in \mathcal{D}^n(\delta_1)$. Hence there exists $m \in \mathbb{Z}^n$ such that $\tilde{y} = \delta_1 m$, and thus $\alpha \tilde{y} = \delta_2 m \in \mathcal{D}^n(\delta_2)$. In turn, $y = \tilde{y} + \delta_1 \varepsilon$ for some $\varepsilon = [\varepsilon_1,\dots, \varepsilon_n]$, $\varepsilon_i \in [-1/2, 1/2)$. It follows that
\begin{equation*}
    \alpha y = \alpha \tilde{y} + \alpha \delta_1 \varepsilon = \delta_2 m + \delta_2 \varepsilon,
\end{equation*}
so $\widetilde{\alpha y} = \delta_2 m$. Thus we show that
\begin{equation}\label{delta_scale}
    \alpha (\tilde{y})_{\delta_1} = (\widetilde{\alpha y})_{\delta_2}.
\end{equation}
Now, by Def.~\ref{roundoff_def}, $\tilde{w}(\tilde{x})_{\delta_1} = (\widetilde{w(\tilde{x})})_{\delta_1}$, so on the basis of Eq.~\eqref{delta_scale} we get that
\begin{equation*}
\alpha \tilde{w}(\tilde{x})_{\delta_1} = (\widetilde{\alpha w(\tilde{x})})_{\delta_2}.
\end{equation*}
Since by Eq.~\eqref{map_scale} we have 
$(\alpha w)(\alpha x) = \alpha w(x)$, it follows that
\begin{equation*}
    (\widetilde{\alpha w(\tilde{x})})_{\delta_2} = (\widetilde{(\alpha w)(\alpha \tilde{x})})_{\delta_2} = \widetilde{(\alpha w)}(\alpha \tilde{x})_{\delta_2} = \widetilde{(\alpha w)}(\widetilde{\alpha x})_{\delta_2},
\end{equation*}
where the last equality follows from Eq.~\eqref{delta_scale}. This completes the proof. 
\end{proof}
\medskip

\medskip
\begin{corollary}
If $w : \Real^n \rightarrow \Real^n$ is a contraction, then for any $\delta_1, \delta_2 > 0$
\begin{equation}\label{scale_minimum}
    \alpha \mathcal{M}[\tilde{w}]_{\delta_1} = \mathcal{M}[\widetilde{\alpha w}]_{\delta_2},
\end{equation}
where $\alpha = \delta_2 / \delta_1$, and $\tilde{w} : \mathcal{D}^n(\delta_1) \rightarrow \mathcal{D}^n(\delta_1)$ and $\widetilde{(\alpha w)} : \mathcal{D}^n(\delta_2) \rightarrow \mathcal{D}^n(\delta_2)$.
\end{corollary}
\medskip
Note that if $w$ is an affine contraction with a fixed point coinciding with the zero vector $\mathbf{0}$ (i.e., $w$ is a linear mapping), then by \eqref{map_scale} we get that $\alpha w = w$ for any $\alpha \neq 0$. Hence, on the basis of Eq.~\eqref{scale_minimum}, any change of the value of $\delta$, which controls the discretization accuracy, has no impact on the cardinality of the resulting minimal absorbing set---in such a case, shifting $\delta$ just scales the set uniformly without any modification of its geometry. In general, for any allowed values $\delta_1$ and $\delta_2$ of the parameter $\delta$, if $w$ is a contraction discretized with respect to $\delta_1$, then there is a contraction $\alpha w$, $\alpha = \delta_2/\delta_1$, whose $\delta_2$-discretization results in a minimal absorbing set of the same cardinality as the minimal absorbing set resulting from the $\delta_1$-discretization of $w$. Hence, for any $\delta_1$ and $\delta_2$, there is a one-to-one correspondence between the sets of contractions determined by a given cardinality of the resulting minimal absorbing sets in $\mathcal{D}^n(\delta_1)$ and $\mathcal{D}^n(\delta_2)$, respectively. This one-to-one correspondence suggests that if contractions are chosen randomly with equal probability, then the chance for the discretization of a contraction to yield a minimal absorbing set of a given cardinality is independent of the value of $\delta$. In the following, we examine this conjecture for the case of affine contractions.   
 
An affine contraction $w : \Real^n \rightarrow \Real^n$ with a fixed point $x_f = w(x_f)$ can be expressed as 
\begin{equation}\label{affine_fixed_point}
    w(x) = l(x - x_f) + x_f,
\end{equation}
where $l : \Real^n \rightarrow \Real^n$ is a contractive linear mapping. On that basis, we can construct a \emph{random} affine contraction on $(\Real^n, d)$ as a pair $[L, X_f]$ of (independent) random variables such that $X_f \in \Real^n$ is a random vector representing a fixed point, and $L \in \{ M\in \Real^{n\times n} : \left\| M \right \|_d < 1\}$ is a random matrix representing a linear contraction on $(\Real^n, d)$, where $\left\| . \right \|_d$ is the matrix norm induced by the vector norm that underlies the metric $d$. We can thus define a random affine contraction on $(\Real^n, d)$ by treating  $[L, X_f]$ as a mapping on $\Real^n$: 
\begin{equation*}
    [L, X_f](x) := L(x - X_f) + X_f.
\end{equation*}
Finally, using the notation above we can define a random variable
\begin{equation*}
    \mathcal{M}^{\#}[L, X_f] := \left| \mathcal{M}\widetilde{[L, X_f]} \right|,
\end{equation*}
whose outcomes are the cardinalities of the minimal absorbing sets for the roundoffs of $[L, X_f]$.

The following lemma formalizes the statement that if for a given $\delta$, the conditional distributions of the random variable $X_f$ over $\delta$-cubes are identical, then we can simplify the study of the distribution of $\mathcal{M}^{\#}[L, X_f]$ by considering the outcomes of $X_f$ restricted to an arbitrarily chosen $\delta$-cube.   
\medskip
\begin{lem}\label{lemma_trans}
Let $\delta >0$ be given. Let $[L, X_f]$ be a random affine contraction on $\Real^n$ such that the probability distribution of $X_f \in \Real^n$ satisfies
\begin{equation}\label{cond_equ}
    \prob(X_f \;|\; X_f \in C_{\delta}(m_1)) = \prob(X_f - \delta(m_2 - m_1) \;|\; X_f \in C_{\delta}(m_2)),
\end{equation}
for any $m_1, m_2 \in \mathbb{Z}^n$ for which $\prob(X_f \in C_{\delta}(m_1)), \prob(X_f \in C_{\delta}(m_2)) \neq 0$, where $C_{\delta}(m_1), C_{\delta}(m_2) \in \mathcal{G}^n(\delta)$. Then 
\begin{equation}
    \prob(\mathcal{M}^{\#}[L, X_f]) = \prob(\mathcal{M}^{\#}[L, X_f] \;|\; X_f \in C_{\delta}(m))
\end{equation}
for any $m \in \mathbb{Z}^n$ such that $\prob(X_f \in C_{\delta}(m)) \neq 0$.
\end{lem}
\begin{proof}
From the assumption \eqref{cond_equ} we get that
\begin{equation}\label{assump}
    \prob(\mathcal{M}^{\#}[L, X_f] \;|\; X_f \in C_{\delta}(m_1)) = \prob(\mathcal{M}^{\#}[L, X_f - \delta (m_2 - m_1)] \;|\; X_f \in C_{\delta}(m_2)).
\end{equation}
In turn, from Eq.~\eqref{affine_fixed_point} it follows that  
\begin{equation}\label{fixed_trans}
(w - \delta m)(x) = l(x - x_f + \delta m) + x_f - \delta m,  
\end{equation}
so $w - \delta m$ is an affine contraction with the fixed point $x_f - \delta m$ and with the same linear part $l$ as in $w$. Moreover, from Eq.~\eqref{trans_minimum} we get that
\begin{equation}\label{card_trans}
    \left|\mathcal{M}[\widetilde{w - \delta m}] \right| = \left|\mathcal{M}[\tilde{w}]\right|
\end{equation}
Setting $m = m_2 - m_1$, from Eqs.~\eqref{fixed_trans} and \eqref{card_trans} we conclude that $\mathcal{M}^{\#}[L, X_f - \delta (m_2 - m_1)] = \mathcal{M}^{\#}[L, X_f]$, and hence
\begin{equation}\label{assum_con}
    \prob(\mathcal{M}^{\#}[L, X_f - \delta (m_2 - m_1)] \;|\; X_f \in C_{\delta}(m_2)) = \prob(\mathcal{M}^{\#}[L, X_f] \;|\; X_f \in C_{\delta}(m_2)).
\end{equation}
Putting Eqs.~\eqref{assump} and \eqref{assum_con} together we get that
\begin{equation*}
    \prob(\mathcal{M}^{\#}[L, X_f] \;|\; X_f \in C_{\delta}(m_1)) = \prob(\mathcal{M}^{\#}[L, X_f] \;|\; X_f \in C_{\delta}(m_2)).
\end{equation*}
Since the cubes in $\mathcal{G}^n(\delta)$ form a countable partition of $\Real^n$, the conclusion of the theorem follows from the law of total probability:
\begin{equation*}
\begin{split}
    \prob(\mathcal{M}^{\#}[L, X_f]) &= \sum_{m \in \mathbb{Z}^n} \prob(\mathcal{M}^{\#}[L, X_f] \;|\; X_f \in C_{\delta}(m)) \prob(X_f \in C_{\delta}(m))\\ 
    &= \prob(\mathcal{M}^{\#}[L, X_f] \;|\; X_f \in C_{\delta}(m)).
\end{split}
\end{equation*}
\end{proof}
The next theorem justifies the approach we took in the numerical experiments, whose results are presented in the sequel. The theorem states that the distribution of the random variable $\mathcal{M}^{\#}[L, X_f]$ in which $X_f$ is drawn from the uniform distribution over a $\delta_1$-cube for an arbitrarily chosen $\delta_1$, is identical to the distribution of any random variable $\mathcal{M}^{\#}[L, Y_f]$ in which $Y_f$ is drawn from the uniform distribution over any bounded set made of a finite number of $\delta_2$-cubes, where $\delta_2$ is any value of the parameter $\delta$. Roughly speaking, if we limit the set of possible affine contractions to the ones with fixed points in a bounded set (which is a natural constraint in real computation) and assume that the occurrence of an affine contraction with any of these fixed points is equally likely, then we can simplify the study of the distribution of $\mathcal{M}^{\#}[L, X_f]$ by considering the outcomes of $X_f$ restricted to an arbitrarily chosen $\delta$-cube for an arbitrarily chosen $\delta$. 
\medskip
\begin{thmm}\label{sampling_thm}
Given $\delta_1, \delta_2 > 0$, let $\mathcal{C}_{\delta_1} \subset \mathcal{G}^n(\delta_1)$ and $\mathcal{C}_{\delta_2} \subset \mathcal{G}^n(\delta_2)$ be finite families of cubes from the $\delta_1$-grid and $\delta_2$-grid in $\Real^n$, respectively. Let $[L, X_f]$ and $[L, Y_f]$ be random affine contractions on $\Real^n$ such that $X_f \in \bigcup \mathcal{C}_{\delta_1}$ and $Y_f \in \bigcup \mathcal{C}_{\delta_2}$ with the uniform probability distributions over both sets. Then, for any $m\in M_{\delta_1} = \{k \in \mathbb{Z}^n : C_{\delta_1}(k) \in \mathcal{C}_{\delta_1} \}$, 
\begin{equation}
    \prob(\mathcal{M}^{\#}[L, Y_f]) = \prob(\mathcal{M}^{\#}[L, X_f] \;|\; X_f \in C_{\delta_1}(m)).
\end{equation}
\end{thmm}
\begin{proof}
Since $Y_f$ is distributed uniformly over $\bigcup \mathcal{C}_{\delta_2}$, it satisfies the premises of Lemma~\ref{lemma_trans} and from the lemma we get that for any $m' \in M_{\delta_2} = \{k' \in \mathbb{Z}^n : C_{\delta_2}(k') \in \mathcal{C}_{\delta_2} \}$, \begin{equation}\label{goal_8}
    \prob(\mathcal{M}^{\#}[L, Y_f]) = \prob(\mathcal{M}^{\#}[L, Y_f] \;|\; Y_f \in C_{\delta_2}(m')).
\end{equation}
Therefore, we have to show that for any $m\in M_{\delta_1}$ and $m'\in M_{\delta_2}$,
\begin{equation*}
    \prob(\mathcal{M}^{\#}[L, Y_f] \;|\; Y_f \in C_{\delta_2}(m')) = \prob(\mathcal{M}^{\#}[L, X_f] \;|\; X_f \in C_{\delta_1}(m)).
\end{equation*}

Let $\alpha = \delta_2 / \delta_1$. For any $m, m' \in \mathbb{Z}^n$ we have
\begin{equation*}
    C_{\delta_2}(m') = \alpha C_{\delta_1}(m) - \delta_2(m - m'),
\end{equation*}
Because $X_f$ and $Y_f$ are uniformly distributed on $\bigcup \mathcal{C}_{\delta_1}$ and $\bigcup \mathcal{C}_{\delta_2}$ respectively, as a consequence we get that for any $m \in M_{\delta_1}$ and $m' \in M_{\delta_2}$, 
\begin{equation}\label{XvsY}
    \prob(Y_f \;|\; Y_f \in C_{\delta_2}(m')) = \prob(\alpha X_f - \delta_2(m - m') \;|\; X_f \in C_{\delta_1}(m)). 
\end{equation}
In turn, from Eq.~\eqref{affine_fixed_point} for any $k \in \mathbb{Z}^n$ we have 
\begin{equation}\label{fixed_scaletrans}
    (\alpha w - \delta k)(x) = l(x - \alpha x_f + \delta k) + \alpha x_f - \delta k,
\end{equation}
so $\alpha w - \delta k$ is a contraction with the fixed point $\alpha x_f - \delta k$ and with the same linear part as $w$. Furthermore, Eqs.~\eqref{trans_minimum} and \eqref{scale_minimum} imply that
\begin{equation*}
    \mathcal{M}[\widetilde{\alpha w - \delta_2 k}]_{\delta_2} = \mathcal{M}[\widetilde{\alpha w}]_{\delta_2} - \delta_2 k = \alpha \mathcal{M}[\tilde{w}]_{\delta_1} - \delta_2 k,
\end{equation*}
and hence
\begin{equation}\label{card_scaletrans}
    \left|\mathcal{M}[\widetilde{\alpha w - \delta_2 k}]_{\delta_2}\right| = \left|\mathcal{M}[\tilde{w}]_{\delta_1} \right|.
\end{equation}
Putting $k = m - m'$, from Eqs.~\eqref{fixed_scaletrans} and \eqref{card_scaletrans} we conclude that $\mathcal{M}^{\#}[L, X_f] = \mathcal{M}^{\#}[L, \alpha X_f - \delta_2 (m - m')]$, and thus for any $m\in M_{\delta_1}$ and $m'\in M_{\delta_2}$, 
\begin{equation*}
\begin{split}
    \prob(\mathcal{M}^{\#}[L, X_f] \;|\; X_f \in C_{\delta_1}(m)) &= \prob(\mathcal{M}^{\#}[L, \alpha X_f - \delta_2 (m - m')]  \;|\; X_f \in C_{\delta_1}(m))\\
    &= \prob(\mathcal{M}^{\#}[L, Y_f] \;|\; X_f \in C_{\delta_2}(m')),
\end{split}
\end{equation*}
where the second equality follows from Eq.~\eqref{XvsY}. 
\end{proof}
\medskip
It is easy to see that the above lemma and theorem also hold when the random variables $\mathcal{M}^{\#}[L,.]$, the cardinalities of minimal absorbing sets, will be replaced with the random variables defined by
\begin{equation*}
    \mathcal{M}_{\#}[L, .] := \textit{the number of components of } \mathcal{M}[L, .]. 
\end{equation*}After the replacement, the proofs of the lemma and theorem remain the same. 

Obviously, the distributions of $\mathcal{M}^{\#}[L, X_f]$ and $\mathcal{M}_{\#}[L, X_f]$ are also related to the distribution of $L$, the random variable for the matrix describing the linear part of an affine contraction. Depending on underlying assumptions, various distributions can be used for this goal. In the numerical experiments presented below, we assumed, just like for fixed points, that the random variable $L$ is distributed uniformly, that is, we posited that each linear contraction has an equal chance to appear. A distribution can be spread over linear contractions on $(\Real^n, d_E)$ by taking advantage of the singular value decomposition $L = U \diag(\lambda_1, \dots, \lambda_n) V$, where $U, V \in \mathrm{O}(n)$, orthogonal matrices of size $n\times n$, and $\diag(\lambda_1, \dots, \lambda_n)$ is a diagonal matrix of singular values\footnote{The contractivity factor $\lambda$ of $L$ with respect to $d_E$ (and, thus, the contractivity factor of an affine mapping $w(.) = L(.) + t$) is equal to the maximum of the singular values, $\lambda = \max_i \lambda_i$.}. This opens the door to constructing a distribution over linear contractions with respect to scaling factors $\lambda_i$ and some parameters related to transformations of rotations and reflections represented by matrices $U$ and $V$. 

In our experiments for affine contractions on $(\Real^2, d_E)$, we generated uniformly distributed random linear contractions using the formula $L = U(\alpha) \diag(\lambda_1, \lambda_2) V(\beta)$, where $U(\alpha), V(\beta) \in \mathrm{SO}(2)$ were rotation matrices parametrized with i.i.d. random angles $\alpha, \beta \sim \mathcal{U}[0, 2\pi)$, and $\lambda_1, \lambda_2 \sim \mathcal{U}(-1, 1)$ were i.i.d. random scaling factors potentially accompanied by random reflections. In turn, random fixed points were generated by uniform sampling $\delta$-cube $C_1(\mathbf{0})$, that is, $X_f \sim \mathcal{U}[-\tfrac{1}{2}, \tfrac{1}{2})^2$. Due to Theorem \ref{sampling_thm}, such a method of sampling generates samples of $[L, X_f]$ which are representative for all distributions of $\mathcal{M}^{\#}[L, X_f]$ (and $\mathcal{M}_{\#}[L, X_f]$) related to affine contractions with $X_f$ distributed uniformly over any finite union of $\delta$-cubes and regardless of the value of $\delta > 0$. In Fig.~\ref{fig:3} and \ref{fig:5} we present plots related to conditional probabilities and conditional expected values of the random variables $\mathcal{M}^{\#}[L, X_f]$ and $\mathcal{M}_{\#}[L, X_f]$. Each value in the plots was computed based on $2\cdot 10^4$ samples drawn from a given conditional distribution.

\begin{figure}[t]
\centering
\begin{subfigure}{.5\textwidth}
  \centering
  \includegraphics[width=1\linewidth]{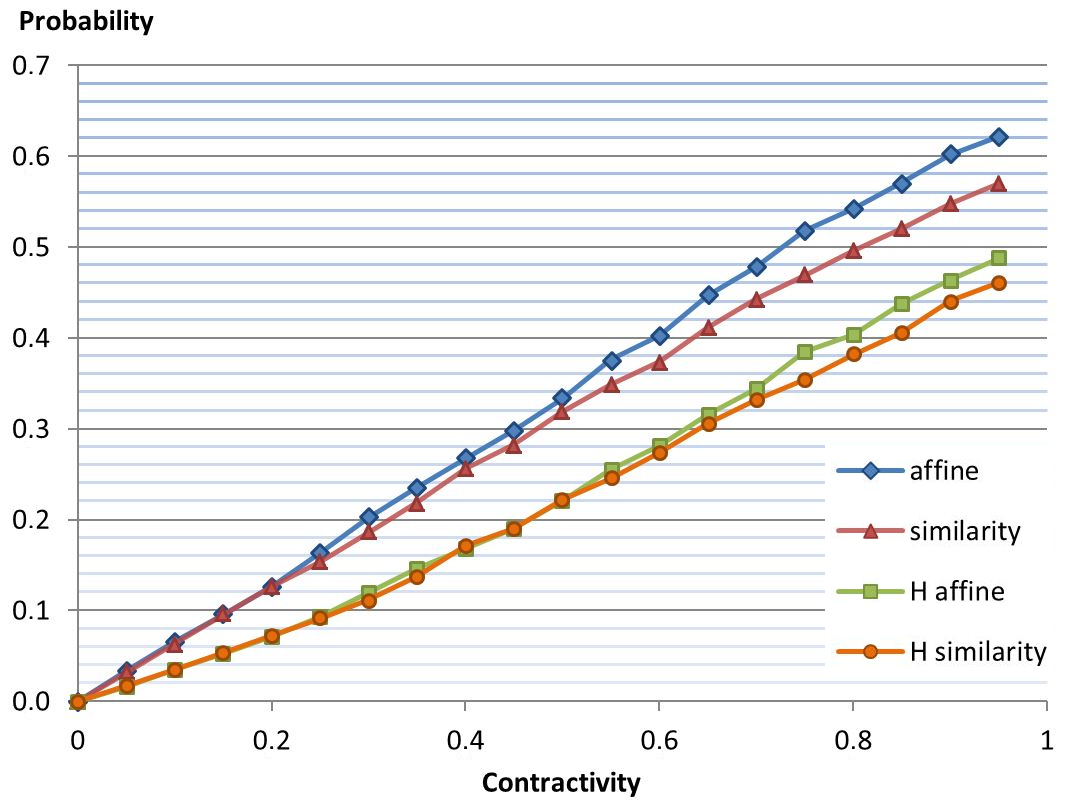}
  \caption{}
  \label{fig:3a}
\end{subfigure}%
\begin{subfigure}{.5\textwidth}
  \centering
  \includegraphics[width=1\linewidth]{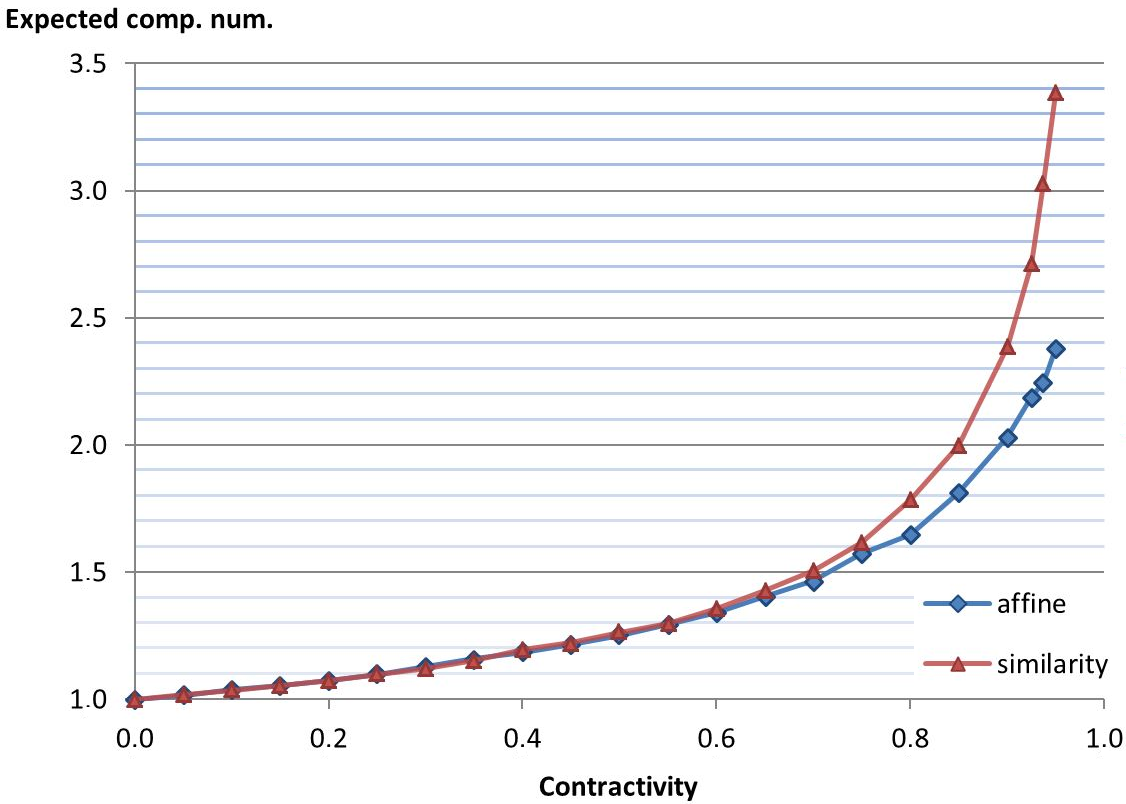}
  \caption{}
  \label{fig:3b}
\end{subfigure}
\caption{Statistics related to a minimal absorbing set of a $\delta$-roundoff of a 2D affine contraction with respect to the contractivity factor less than a given value (see the main text): (a) the probability for the minimal absorbing set generated by an affine (blue) and similarity (red) contraction to be a non-singleton and to have more than one component (green and orange plots, respectively); (b) the expected value of the number of the a minimal absorbing set components}\label{fig:3}
\end{figure}

\begin{figure}[t]
\centering
\begin{subfigure}{.5\textwidth}
  \centering
  \includegraphics[width=1\linewidth]{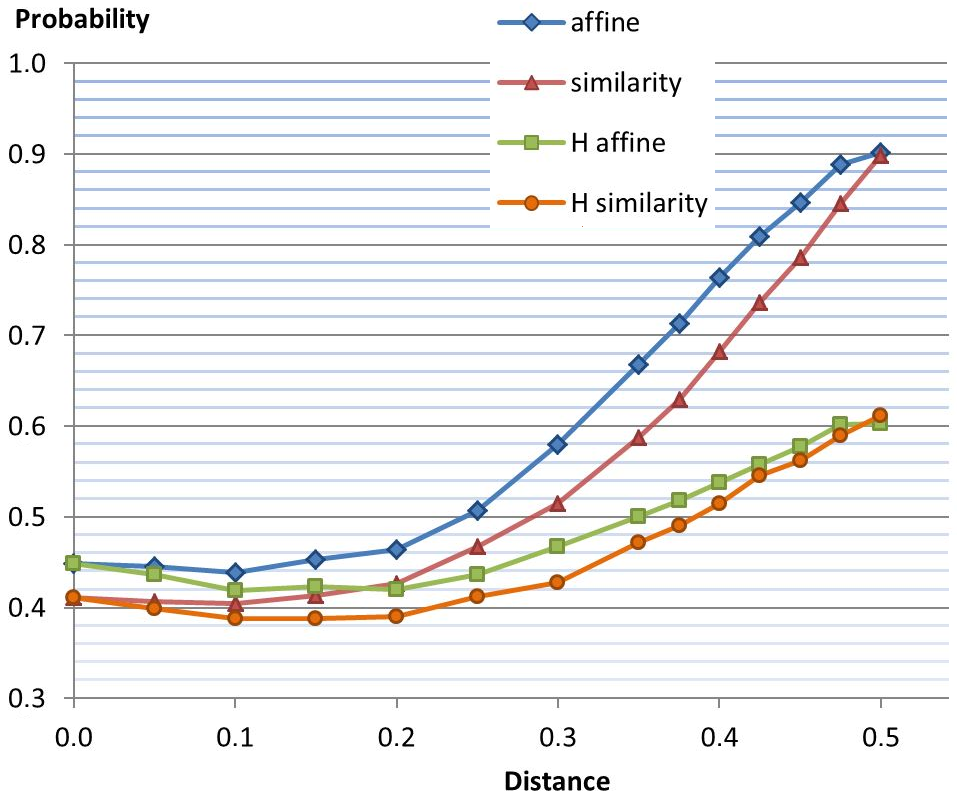}
  \caption{}
  \label{fig:5a}
\end{subfigure}%
\begin{subfigure}{.5\textwidth}
  \centering
  \includegraphics[width=1\linewidth]{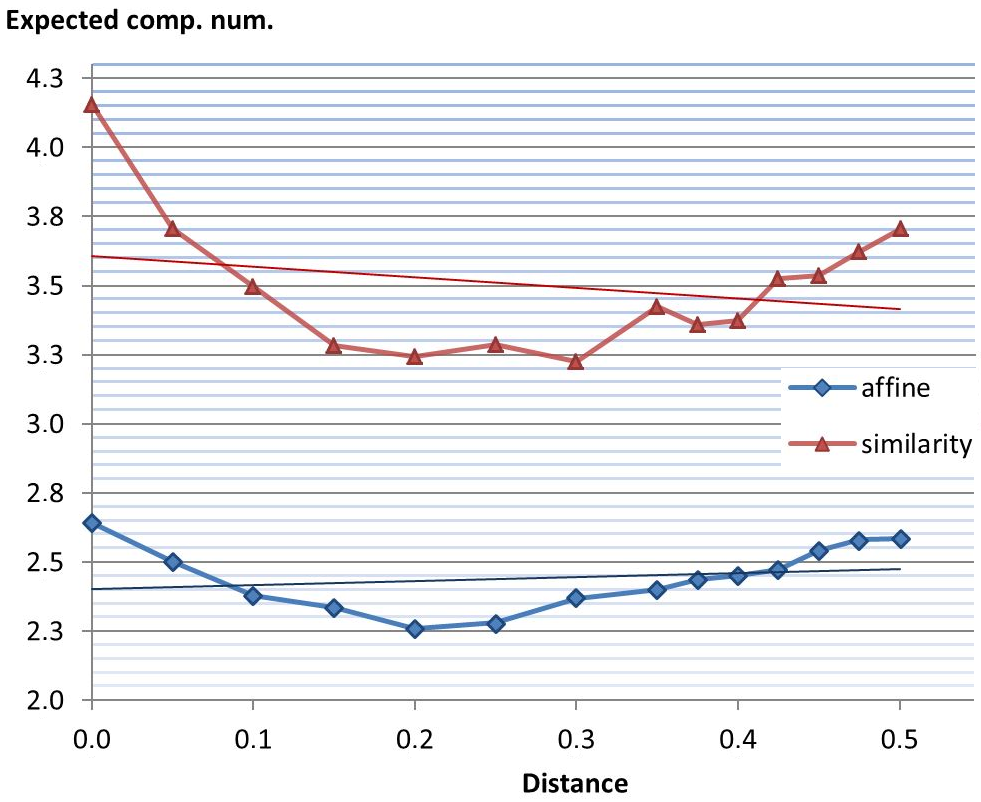}
  \caption{}
  \label{fig:5b}
\end{subfigure}
\caption{Same as in Fig.~\ref{fig:3}, but with respect to the $d_{\infty}$ distance of the fixed point of a mapping to the center of the $\delta$-cube (see the main text). In (b) additionally linear trends are shown}\label{fig:5}
\end{figure}

Fig.~\ref{fig:3a} shows plots of the conditional probabilities for an affine contraction and contractive similarity to lose contractivity and to have more than one minimal absorbing set component, with respect to a parameter $s \in [0, 1)$ that determines the upper bound for the contractivity factors of the mappings. The plots depict the following conditional probabilities: for affine contractions, in blue $\prob(\mathcal{M}^{\#}[L,X_f] > 1 | \max(\lambda_1,\lambda_2) \leq s)$ and in green  $\prob(\mathcal{M}_{\#}[L,X_f] > 1 | \max(\lambda_1,\lambda_2) \leq s)$; for contractive similarities, in red $\prob(\mathcal{M}^{\#}[L,X_f] > 1 | \lambda_1 = \lambda_2 \leq s)$ and in orange $\prob(\mathcal{M}_{\#}[L,X_f] > 1 | \lambda_1 = \lambda_2 \leq s)$. It is clearly seen that the probability of losing the contractive property by a mapping due to discretization is nearly proportional to the contractivity factor, and reaches the value about $0.6$ for the maps with Lipschitz constants not exceeding the value of $0.95$, both for affine and similarity contractions. Moreover, the probability is greater than half for affine contractions with contractvity factors in the range of $[0, 0.75]$. First of all, this confirms the statement we made at the beginning of this section that the probability of losing contractivity by an affine contraction on $(\Real^n, d_E)$ due discretization is higher than the probability that the map stays contractive. Secondly, according to the definition of a \emph{hyperbolic} IFS, all its maps are assumed to be contractions, and hence, in light of the presented results, we can conclude that the probability of losing the hyperbolic property by an affine IFS on $(\Real^2, d_E)$ due discretization grows exponentially with the number $N$ of IFS mappings as $1 - 0.4^N$ or so. 
As one can see in the figure, similar nearly linear plots were obtained for the probability that the number of components of a minimal absorbing set is greater than one (the green and orange plots). As shown in the next section this condition is essential for the number of attractive invariant sets that a discrete IFS may generate. In turn, Fig.~\ref{fig:3b} shows plots of conditional expected values of the number of components of minimal absorbing sets for $\delta$-roundoffs of affine and similarity contractions with contractivity factors in the parametrized range $[0, s)$, $s \in [0, 1)$, that is, the plots present the values of  $\mathbf{E}(\mathcal{M}_{\#}[L,X_f] | \max(\lambda_1,\lambda_2) \leq s)$ and $\mathbf{E}(\mathcal{M}_{\#}[L,X_f] | \lambda_1 = \lambda_2 \leq s)$. The first thing to note is a rapid increase of the functions that begins about the value of $0.65$ of the contractivity upper bound. 
The second thing is that the increase is more rapid for similarities than affine contractions. 

In Fig.~\ref{fig:5} we present plots which depict the values being counterparts of the previous ones but this time conditioning was with respect to the $d_{\infty}$ distance of the fixed point of a mapping to the center of a $\delta$-cube. That is, Fig.~\ref{fig:5a} shows the following conditional probabilities: for affine contractions, in blue $\prob(\mathcal{M}^{\#}[L,X_f] > 1 | d_{\infty}(X_f, \mathbf{0}) = d)$ and in green $\prob(\mathcal{M}_{\#}[L,X_f] > 1 | d_{\infty}(X_f, \mathbf{0}) = d)$; for contractive similarities, in red $\prob(\mathcal{M}^{\#}[L,X_f] > 1 | d_{\infty}(X_f, \mathbf{0}) = d, \; \lambda_1 = \lambda_2)$ and in orange $\prob(\mathcal{M}_{\#}[L,X_f] > 1 | d_{\infty}(X_f, \mathbf{0}) = d, \; \lambda_1 = \lambda_2)$, where $d \in [0, 0.5]$ represents a distance-to-center parameter. In addition, for numerical stability reasons, in computation we restrict the possible values of contractivity factors of sampled contractions to the range $[0, 0.95]$. In the plots it is clearly seen that, as of the value of $0.2$ of the distance from the center, all probabilities start to grow more rapidly. Another interesting thing, perceivable in all four plots, is that the probabilities seem to decrease in the range $[0, 0.1]$ of the distance. Some other manifestations of these facts are also seen in Fig.~\ref{fig:5b} in which we present the expected values of the number of minimal absorbing set components for affine and similarity contractions, conditioned with the distance, that is, the plots of $\mathbf{E}(\mathcal{M}_{\#}[L,X_f] | d_{\infty}(X_f, \mathbf{0}) = d)$ and  $\mathbf{E}(\mathcal{M}_{\#}[L,X_f] | d_{\infty}(X_f, \mathbf{0}) = d, \; \lambda_1 = \lambda_2)$, respectively. According to the plots, on average the largest number of components is generated by contractions with fixed points coincident with the center of a $\delta$-cube, and that this number decreases as fixed points move away from the center up to the value $0.2$ of the distance. It is also clearly visible that similarity contractions generate on average almost two times as many components as affine contractions---an additional aspect of the fact we have already noticed while discussing Fig.~\ref{fig:3b}. And last but not least, it is interesting that the average number of components is nearly constant---it fluctuates only within the limits of one compound for similarities and half of a compound for affine contractions, and relatively small---the linear trend is about the value of $3.5$ for similarites and about the value of $2.4$ for affine contractions. 

\section{Discrete Iterated Function System}\label{Sec_DIFS}
In this section we will study invariant sets and measures that arise from iterated function systems acting on a discrete space $\mathcal{D}^n(\delta)$. We will define a discrete IFS in a place-dependent probability version, that is, the probabilities assigned to IFS mappings are real-valued functions rather than constant numbers.
\medskip
\begin{definition}\label{def_DIFS}
Let $S$ be a subset of $\mathcal{D}^n(\delta)$ endowed with a metric $d$. Let $\{ \tilde{w}_i\}_{i=1}^N$ be a finite set of maps $\tilde{w}_i : S \rightarrow S$. In addition, associate with each map $\tilde{w}_i$ a function $p_i : S \rightarrow (0, 1]$ of place-dependent, positive probability weights such that $\sum_{i=1}^N p_i(\tilde{x}) = 1$ for every $\tilde{x} \in S$. We will refer to the set $\{ S; \tilde{w}_1, \dots, \tilde{w}_N; p_1,\dots, p_N\}$ as a \emph{discrete iterated function system with place-dependent probabilities} (DIFS for short). 
\end{definition}
\medskip
Hereafter, unless otherwise stated, when speaking about a DIFS we will mean a discrete iteration function system with place-dependent probabilities, whereas an IFS will always appear in the constant probabilities and, moreover, hyperbolic version in this paper. Also note that in the definition above we do not impose any metric properties on the DIFS mappings, in particular we do not require them to be contractions. For now, they are just mappings on a subset $S$ of a discrete space $\mathcal{D}^n(\delta)$, and the only metric properties they own are those that follow from the properties of $S$ itself, namely, the mappings are continuous (with respect to any metric, since $S$ is discrete) and if $S$ was additionally bounded, they would be Lipschitz.   

The random iteration algorithm (RIA) in the version for a discrete IFS with place-dependent probabilities works as follows: Choose a point $\tilde{x}_0 \in S$ and generate a random orbit $\{\tilde{x}_k = \tilde{X}_k(\tilde{x}_0)\}_{k=0}^\infty$ by recursively setting  $\tilde{X}_k(\tilde{x}_0) := \tilde{w}_{I_k}(\tilde{X}_{k-1}(\tilde{x}_0))$, where $\tilde{X}_0(\tilde{x}_0) = \tilde{x}_0$ and each mapping index $I_k$ is drawn from the probability distribution on $\{1, \dots, N \}$,  specified by the values of the $N$ probability weight functions at the "place" $\tilde{x}_{k-1} = \tilde{X}_{k-1}(\tilde{x}_0)$, that is, the distribution is $[p_1(\tilde{x}_{k-1}), \dots, p_N(\tilde{x}_{k-1})]$. In a special case when the probability functions $p_i$ are constant over $S$, and hence the random variables $I_k$ are i.i.d. with a distribution independent of place, the DIFS with place-dependent probabilities reduces to a discrete IFS with constant probabilities. 

On the basis of above, RIA can be considered in terms of a collection of random variables $\{\tilde{X}_k : k \geq 0\}$ that forms a stochastic process on a countable set $S$, where $\tilde{X}_0$ is distributed with a point mass distribution concentrated at $\tilde{x}_0$ (i.e., the Dirac measure $\delta_{\tilde{x}_0}$).
By the recursive definition of the random variables $\tilde{X}_k$ we get that, for any $k \geq 0$, $\tilde{X}_k$ depends only on $\tilde{X}_{k-1}$, that is, for any $\tilde{x}, \tilde{y} \in C$ and $k \geq 0$,
\begin{equation*}
    \prob(\tilde{X}_k = \tilde{y} \;|\; \tilde{X}_{k-1} = \tilde{x}, \dots, \tilde{X}_{1}, \tilde{X}_{0}) = \prob(\tilde{X}_k = \tilde{y} \;|\; \tilde{X}_{k-1} = \tilde{x}),
\end{equation*} 
and further that
\begin{equation}\label{prob_entry}
    \prob(\tilde{X}_k = \tilde{y} \;|\; \tilde{X}_{k-1} = \tilde{x}) = \sum_{i=1}^N p_i(\tilde{x}) \;\mathbf{1}_{\{\tilde{y}\}}(\tilde{w}_i(\tilde{x})),
\end{equation}
where $\mathbf{1}_A(.)$ denotes the indicator function of a set $A$. Therefore, the stochastic process is a (time-homogeneous) Markov chain with the transition matrix $P$ whose entries $(P)_{\tilde{x}\tilde{y}} = \prob(\tilde{X}_k = \tilde{y} \;|\; \tilde{X}_{k-1} = \tilde{x})$ are specified by Eq.~\eqref{prob_entry}. 

In the following definition and theorem we gather some basic facts and notions pertaining to Markov chains that we will use further on (for more rigorous treatment see e.g. \cite{Serf09,Stro05}). 
\medskip
\begin{definition}\label{Markov_def}
Let $\{\tilde{X}_k\}$ be a Markov chain on a countable set $S$, called the \emph{state space} of the chain, and a transition matrix $P$.\\
\emph{(a)} A state $\tilde{y}$ is said to be \emph{accessible} from a state $\tilde{x}$ if there is a non-zero probability of reaching $\tilde{y}$ in a finite number of steps when starting from $\tilde{x}$:
\begin{equation*}
\exists k \in \Natural,\; \prob(\tilde{X}_k = \tilde{y} \;|\; \tilde{X}_0 = \tilde{x}) = (P^k)_{\tilde{x}\tilde{y}} > 0. 
\end{equation*}
If $\tilde{x}$ and $\tilde{y}$ are accessible from each other, then they are said to \emph{communicate}. 
\medskip
\newline
\emph{(b)} A Markov chain is called \emph{irreducible} if any couple of its states communicate.  
\medskip
\newline
\emph{(c)} A maximum subset $C$ of $S$ such that any two states in $C$ communicate forms a \emph{communication class}. A subset $C$ is a \emph{closed class} if the only states which are accessible from the states in $C$ are those in $C$ (i.e., if the chain starts in $C$, then it will stay in $C$ with probability $1$). Therefore, the Markov chain on any of its closed communication classes can be viewed as irreducible. 
\medskip
\newline
\emph{(d)} A state $\tilde{x}$ is \emph{recurrent} if, with probability $1$, the chain starting in $\tilde{x}$ returns to $\tilde{x}$ in a finite number of steps. A recurrent state is \emph{positive recurrent} if the expected number of steps to return is finite. A state that is not recurrent is called \emph{transient}.  
\end{definition}
\medskip
\begin{thmm}\label{Markov_thm}
The state space $S$ of a Markov chain uniquely decomposes into a countable union of disjoint subsets:
\begin{equation*}
    S = \mathcal{T} \cup \mathcal{A}_1 \cup \mathcal{A}_2 \dots 
\end{equation*}
where $\mathcal{T}$ is the set of all transient states, and $\mathcal{A}_k$ are closed communication classes of the recurrent states. Moreover, if $S$ is finite, then the union 
\begin{equation}\label{recur_states}
\mathcal{A} = \bigcup_{k\leq | S|} \mathcal{A}_k 
\end{equation}
is nonempty and consists of positive recurrent states. In addition, $\mathcal{A}$ is reached by the Markov chain in a finite number of steps with probability one and independently of the initial state.
\end{thmm}
\medskip
Since all states in each set $\mathcal{A}_k$ are either recurrent or positive recurrent, the classes are called \emph{recurrent communication classes} and \emph{positive recurrent communication classes}, respectively\footnote{It is a basic fact of the theory of Markov chains that if a recurrent communication class includes a positive recurrent state, then all states of the class are positive recurrent, so the class is a positive recurrent communication class. Therefore the state space of an irreducible Markov chain consists only of states of the same type: either recurrent or positive recurrent or transient.}. Similarly, a chain on a recurrent communication class and, respectively, positive recurrent communication class is called an \emph{irreducible recurrent chain} and an \emph{irreducible positive recurrent chain}, respectively. 

The next theorem provides the view of closed and communication classes subjected to the action of the Hutchinson operator $\widetilde{W}(.) := \bigcup_{i = 1}^N \tilde{w}_i(.)$ associated with a DIFS:
\medskip
\begin{thmm}\label{Markov2Hutchinson}
Let $\{ S; \tilde{w}_1, \dots, \tilde{w}_N; p_1,\dots, p_N\}$ be a DIFS, and let $\{\tilde{X}_k\}$ be the associated Markov chain. 

\noindent\emph{(a)} $C$ is a closed class of $\{\tilde{X}_k\}$ if and only if $\bigcup_{i=1}^N \tilde{w}_i(C) \subset C$. 

\noindent\emph{(b)} If $C$ is a communication class of $\{\tilde{X}_k\}$, then $\bigcup_{i=1}^N \tilde{w}_i(C) \supset C$.
\end{thmm}
\begin{proof}
(a) Suppose $C$ is a closed class, so for every $i\in \{1,\dots, N \}$, $\tilde{w}_i(C) \subset C$. Hence $\bigcup_{i=1}^N \tilde{w}_i(C) \subset C$. Now, suppose that $\bigcup_{i=1}^N \tilde{w}_i(C) \subset C$. The mappings $w_i$ map $C$ into itself, so for any $\tilde{x} \in C$ and every finite sequence $i_1,\dots, i_m$ of indices from $\{1, \dots, N\}$, we have $w_{i_m}\circ\dots\circ w_{i_1}(\tilde{x}) \in C$. Therefore, if $\tilde{y} \notin C$, then $\tilde{y}$ is not accessible from $C$, and thus $C$ is a closed class. (b) Let $\tilde{x} \in C$, where $C$ is a communication class. Then $\tilde{x}$ is accessible from any state in $C$, and hence there exist at least one $\tilde{y} \in C$ and $i\in \{1,\dots, N \}$ such that $\tilde{w}_i(\tilde{y}) = \tilde{x}$. Hence, $\tilde{x} \in \tilde{w}_i(C)$, and as a consequence $\bigcup_{i=1}^N \tilde{w}_i(C) \supset C$. 
\end{proof}

Theorem \ref{Markov_thm} plays a similar role for DIFS as Banach's fixed point theorem for hyperbolic IFSs. In its first part it states that the presence of a recurrent state in a Markov chain implies that the chain possesses at least one set $\mathcal{A}$, which in light of Theorem \ref{Markov2Hutchinson} is an invariant set of the DIFS Hutchinson operator. The second part provides a sufficient condition for the existence of a recurrent state, and thus also the set $\mathcal{A}$, and moreover it asserts that the set attracts orbits of the chain, in the sense that, with probability one, an orbit eventually visits and stays in $\mathcal{A}$, regardless of the orbit's intial point. Clearly, the properties of $\mathcal{A}$: the (forward) invariance under the DIFS Hutchison operator and the almost sure orbit attraction (if present) very much resemble the behavior of a hyperbolic IFS attractor. The key difference is that, in general, $\mathcal{A}$ decomposes into a countable collection of disjoint subsets $\mathcal{A}_k$, each of which is a fixed point of the DIFS Hutchinson operator and may act as an attractor within its own basin of attraction. This is a different situation from the one of a hyperbolic IFS in that, by Banach's fixed point theorem, the latter always possesses a single attractor being a unique fixed point of the associated Hutchison operator. 
\medskip
\begin{remark}
We have already had a glimpse of recurrent communication classes of Markov chains in Sec.~\ref{sec_minabset}. Indeed, the components of a minimum absorbing set can be treated as recurrent communication classes of a trivial Markov chain that is generated by a single map (and hence the probability assigned to the map is $1$) in a discrete space or its subset. Thereby, we have an example of a DIFS that consists of a single map and can possesses multiple attractors, even though---as we have seen in Sec.~\ref{sec_minabset}---the map is a $\delta$-roundoff of a contraction and, thus, the DIFS itself is a discretized version of a hyperbolic IFS.     
\end{remark}
\medskip
At this point, a natural question arises about upper bounds for the number of such invariant sets that a DIFS may possess.
\medskip
\begin{lem}\label{rec2comp}
Let $\{ S; \tilde{w}_1, \dots, \tilde{w}_N; p_1,\dots, p_N\}$ be a DIFS. Suppose that for at least one DIFS mapping $\tilde{w}_i$ there exists a minimal absorbing set $\mathcal{M}[\tilde{w}_i, S]$. Let $\mathcal{A}_k$ be a recurrent communication class of the associated Markov chain $\{ \tilde{X}_k\}$. If there exists $\tilde{x} \in S$ such that both $\tilde{x} \in \mathcal{A}_k$ and $\tilde{x} \in \mathcal{B}[\mathcal{M}_j[\tilde{w}_i, S]]$, then $\mathcal{A}_k \supset \mathcal{M}_j[\tilde{w}_i, S]$.
\end{lem}
\begin{proof}
Suppose that $\tilde{x} \in S$ satisfies the assumptions above. Since $\tilde{x}$ is in the basin of attraction of $\mathcal{M}_j[\tilde{w}_i, S]$, by Def.~\ref{minabsorbing_def} there is $m\in \Natural$ such that $\tilde{w}_i^{\circ m}(\tilde{x}) \in \mathcal{M}_j[\tilde{w}_i, S]$. Moreover, $\mathcal{M}_j[\tilde{w}_i, S]$ is a periodic orbit for ${w}_i$ in $S$ (Corollary \ref{component_periodic}). Therefore, writing $p \in \Natural$ for the period of $\mathcal{M}_j[\tilde{w}_i, S]$, we get that for any $y \in \mathcal{M}_j[\tilde{w}_i, S]$, there is a certain $k < m + p$ such that 
\begin{equation*}
    \prob(\tilde{X}_k = \tilde{y} \;|\; \tilde{X}_0 = \tilde{x}) \geq \prod_{l=1}^k p_i(\tilde{w}_i^{\circ l}(\tilde{x})) > 0,
\end{equation*}
because $p_i(.)$ is strictly positive over $S$. Hence, every point of $\mathcal{M}_j[\tilde{w}_i, S]$ is accessible from $\tilde{x}$. Since, by assumption, $\tilde{x}$ is also in $\mathcal{A}_k$ and the set is a closed class, as a result we get that $\mathcal{A}_k \supset \mathcal{M}_j[\tilde{w}_i, S]$ as required.  
\end{proof}
\medskip
\begin{thmm}\label{attractor_num_thm}
Let $\{ S; \tilde{w}_1, \dots, \tilde{w}_N; p_1,\dots, p_N\}$ be a DIFS. Let $I \subset \{1, \dots, N \}$ be the subset of the indices of the DIFS maps for which minimal absorbing sets exist. Suppose that $I$ is nonempty and define $\mathcal{M}_{\mathcal{F}} := \{ \mathcal{M}[\tilde{w}_i, S]\}_{i\in I}$. Let $\mathcal{A}_{\mathcal{F}} := \{ \mathcal{A}_k\}$ be the family of the recurrent communication classes of the associated Markov chain. Then the following statements hold:\\
\emph{(a)} The number of the sets in $\mathcal{A}_{\mathcal{F}}$ is not greater than the number of components of any set in $\mathcal{M}_{\mathcal{F}}$, that is, 
\begin{equation*}
    |\mathcal{A}_{\mathcal{F}}| \leq \min_{i\in I} \mathcal{M}_{\#}[\tilde{w}_i, S],
\end{equation*}
where the subscript '$\#$' stands for the number of components of a minimal absorbing set. 
\medskip
\newline
\emph{(b)} If for some $i, j \in I$,
\begin{equation*}
    \mathcal{M}[\tilde{w}_i, S] \subset \mathcal{B}[\mathcal{M}_k[\tilde{w}_j, S]]
\end{equation*}
for a certain $k \in \{1,\dots,\mathcal{M}_{\#}[\tilde{w}_j, S] \}$, then
\begin{equation*}
    | \mathcal{A}_{\mathcal{F}}| \leq 1. 
\end{equation*}
\end{thmm}
\begin{proof}
(a) Let $\mathcal{A}_k \in \mathcal{A}_{\mathcal{F}}$ and let $\mathcal{M}[\tilde{w}_i, S] \in \mathcal{M}_{\mathcal{F}}$. The basins of attraction $\mathcal{B}[\mathcal{M}_j[\tilde{w}_i, S]]$, $j\in \{1,\dots, \mathcal{M}_{\#}[\tilde{w}_i, S]\}$ forms a countable partition of $S$. Therefore, $\mathcal{A}_k$ is a subset of a union of a certain number of the basins, $\mathcal{A}_k = \mathcal{A}_k \cap \bigcup_j\mathcal{B}[\mathcal{M}_j[\tilde{w}_i, S]]$. On the basis of Lemma \ref{rec2comp}, if $\mathcal{A}_k \cap \mathcal{B}[\mathcal{M}_j[\tilde{w}_i, S]] \neq \varnothing$, then $A_k \supset \mathcal{M}_j[\tilde{w}_i, S]$. Hence, each set in $\mathcal{A}_{\mathcal{F}}$ includes at least one component of the set $\mathcal{M}[\tilde{w}_i, S]$. Moreover, for any $\mathcal{A}_{m} \in \mathcal{A}_{\mathcal{F}}$ such that $A_m \supset \mathcal{M}_j[\tilde{w}_i, S]$, we have $\mathcal{A}_k = \mathcal{A}_m$, because the recurrent communication classes are disjoint. Hence, for any $\mathcal{M}[\tilde{w}_i, S] \in \mathcal{M}_{\mathcal{F}}$, the number of sets in $\mathcal{A}_{\mathcal{F}}$ cannot exceed the number of the components of $\mathcal{M}[\tilde{w}_i, S]$, which completes this part of the proof.
\medskip
\newline
(b) If $\mathcal{A}_{\mathcal{F}}$ is empty, the conclusion of the theorem follows trivially. Assume that $\mathcal{A}_{\mathcal{F}}$ is nonempty. The basins of attractions of the components of the set $\mathcal{M}[\tilde{w}_i, S]$ constitute a countable partition of $S$, so by Lemma \ref{rec2comp}, each $\mathcal{A}_k \in \mathcal{A}_{\mathcal{F}}$ includes at least one of the components of $\mathcal{M}[\tilde{w}_i, S]$. But by the assumption of the theorem, all the components belong to the basin of attraction $\mathcal{B}[\mathcal{M}_k[\tilde{w}_j, S]]$ and thus, again by Lemma \ref{rec2comp}, $\mathcal{M}_k[\tilde{w}_j, S]$ is a subset of every set $\mathcal{A}_k \in \mathcal{A}_{\mathcal{F}}$. Since $\mathcal{A}_{\mathcal{F}}$ is a family of disjoint sets, $\mathcal{A}_{\mathcal{F}}$ consists of a single set as required.   
\end{proof}

\subsection{On the existence of multiple attractors of a discretized hyperbolic IFS in practice}\label{on_multiattractor}
In Sec.~\ref{sect_more_often} we showed that discretization of a hyperbolic IFS may lead to losing the contractive property characterizing the original and we also gave numerical evidence that from the statistical point of view, at least in the case of affine contractions on $(\Real^2, d_E)$, such a situation can be considered typical and is not cured by increasing the precision of computation (expressed by the reciprocal of $\delta$). As a consequence, even though we know that because of the fixed-point theorem the original hyperbolic IFS has an attractor and that this attractor is unique, we cannot use the theorem to show that the same facts hold for a discretized version of the IFS. Moreover, Theorem \ref{Markov_thm} postulates, quite the opposite to Banach's theorem, that a discretized version of the hyperbolic IFS may possess more than one invariant set that attracts orbits generated by RIA, and Theorem \ref{attractor_num_thm} (a) provides an upper bound for the number of the invariant sets in terms of the number of components of the minimal absorbing sets, which---as we know from Corollary \ref{contraction_minimumset}---$\delta$-roundoffs of contractions always possess. On the other hand, the second part of Theorem \ref{attractor_num_thm} gives a sufficient condition for a DIFS attractor (if it exists, and we show in the sequel that it is true) to be unique: \textit{At least one of the minimal absorbing sets should be totally included in a basin of attraction of one of the remaining minimal absorbing sets}. Therefore conversely, a necessary condition for the DIFS to have more than one attractor is: \textit{Each minimal absorbing set has to be intersected with at least two basins of attraction of every other minimal absorbing set}. Let us have a look at the possibility of the occurrence of such an event in practice. 

Theorem \ref{thm_alpha_eps} gives us an upper bound on the diameter of a minimal absorbing set for a contraction and if we treat the value of $\diam_d(C_{\delta})$ as a unit of length, it is clearly seen that the upper bound for the size of the minimal absorbing set depends only on the contractivity factor of the original contraction and remains constant when the precision of computation is changed. In this setting, altering the precision of computation affects only the unit of length in the way that the larger precision (or equivalently, the smaller $\delta$), the smaller the unit of length. In effect, increasing the precision makes the extents of the minimal absorbing sets get smaller relative to the distances between the sets. At the same time, the basins of attractions spread radially from the minimal absorbing sets and their homogeneous regions get larger and larger as the distance from the minimal absorbing set increases (cf. the figures in the previous sections). Consequently, the higher precision of computation, the higher probability that the minimal absorbing set sits entirely in one of the basins of attraction of another minimal absorbing set. For resolutions of computation used in practice such as those offered by floating point arithmetic the minimal absorbing sets are usually so tiny relative to homogeneous regions of the basins of attraction that the configuration in which each of the sets is intersected by more than one basin of attraction of every other set is very special, not to say highly improbable. Therefore, even though a DIFS version of a hyperbolic IFS is usually not hyperbolic regardless of precision of computation in use but if the precision is reasonably high, then the DIFS not only has an attractor (we show it later on) but also this attractor is typically unique. Therefore the uniqueness of the DIFS attractor is not a result yielded by Banach's fixed point theorem (as it is sometimes suggested more or less explicitly in literature), but merely a typical resultant of the geometry of minimal absorbing sets and their basins of attraction when "embedded" in a suitable resolution of computation. 

\subsection{Invariant measure and stationary distribution}\label{sec_statDist}
Now we consider the behavior of an orbit generated by RIA within a recurrent communication class. Suppose that a DIFS possesses at least one recurrent communication class $\mathcal{A}_k$ and assume that the associated Markov chain arrived in the set at a certain step while wandering through $S$ or it just started in one of the points of the set. Because $\mathcal{A}_k$ is a closed communication class the orbit will never leave the set and the Markov chain on $\mathcal{A}_k$ is irreducible (Def.~\ref{Markov_def} (c)). In the context of the problem of approximating the geometry of $\mathcal{A}_k$ it is crucial whether the orbit visits all points of the set during the evolution of the chain. It is not hard to show (see e.g. \cite{Fell68}, p.~391) that if the chain is irreducible recurrent than for any state $\tilde{x} \in \mathcal{A}_k$,
\begin{equation}\label{recurrent_orbit_visit}
    \prob(\exists i \in \Natural : \tilde{X}_i = \tilde{x}) = 1
\end{equation}
regardless of the initial state $\tilde{X}_0 = \tilde{x}_0 \in  \mathcal{A}_k$. In other words, starting from any arbitrary state in $\mathcal{A}_k$, the orbit is certain to pass through every other state in $\mathcal{A}_k$. 

Moreover, the orbits of the Markov chain on a recurrent communication class possess some \emph{average} statistical properties reflected in an \textit{invariant measure} supported by the set. Writing $P$ for the transition matrix of the chain restricted to a recurrent communication class $\mathcal{A}_k$ (an irreducible recurrent chain), an invariant measure $\pi$ can be represented by a row vector $[\pi] = [(\pi)_{\tilde{x}}]_{\tilde{x}\in \mathcal{A}_k}$ of the values $(\pi)_{\tilde{x}} = \pi\{\tilde{x}\}$, which satisfies the equation $[\pi] = [\pi] P$. It can be shown (see \cite{Fell68,Serf09}, also \cite{Meye00} for the treatment from the standpoint of Perron-Frobenius theory) that such a vector for a recurrent communication class always exists, its entries are positive and finite, and equal to the \emph{expected} number of visits the chain makes to states $\tilde{x}$ while returning to a state $\tilde{y} \in \mathcal{A}_k$. Because of the relationship between the values of $\pi$ and a given baseline state $\tilde{y}$, there can be, in general, more than one invariant measure for a recurrent communication class, however the measures are equal up to multiplication by a constant. Another thing is that the value of the invariant measures on a infinite recurrent communication set can be infinite.
But if an invariant measure $\pi$ is finite, then after normalization becomes a probability measure referred to as a \emph{stationary distribution}. It is known from from the Markov chain theory that the normalization of an invariant measure on a recurrent communication class is possible if and only if the class is \emph{positive} recurrent. (As asserted in the second part of Theorem \ref{Markov_thm}, this positive recurrence condition is always satisfied for finite closed communication classes.) The values of the stationary distribution entries $(\pi)_{\tilde{x}}$ are reciprocals of the expected number of steps to return, they are finite and characterize each positive recurrent state $\tilde{x}$ (Def.~\ref{Markov_def} (d)). From this follows that a stationary distribution of a Markov chain restricted to a positive recurrent communication class must be unique. Clearly, if for a certain $i \geq 0$, the distribution of $\tilde{X}_i$ of a chain $\{\tilde{X}_k\ : k \geq 0\}$ is a stationary distribution $\pi$, then the distributions of the random variables that follow in $\{\tilde{X}_k\ : k \geq i\}$ do not change, they are identical\footnote{For any $k > 0$, the distribution $\mu_k$ of $\tilde{X}_k$ is given by $[\mu_{k}] = [\mu_{k-1}] P$, where $\mu_{k-1}$ is the distribution of $\tilde{X}_{k-1}$ (Eq.~\eqref{prob_entry}).} and equal to $\pi$---the process is said to be in a steady (or equilibrium) state. 

\subsection{To be attractive (or at least ergodic)}\label{sec_be_attractive}
It is known (see \cite{Barn88}, Theorem 2.1; also \cite{Diac99,Sten02}) that if an IFS is hyperbolic on $(\Real^n, d)$ or even only contractive on average\footnote{In the relevant literature it is common to weak the contractivity on average by a seemingly less restrictive condition of the contractivity on "log-average". As noticed in \cite{Sten12}, however, this improvement is more of cosmetic nature, because the two conditions are equivalent after an appropriate change of metric.} (with some additional constraints on place dependent $p_i$'s such as Dini's continuity) the stationary distribution is unique and attractive\footnote{Actually, the result holds for much more general spaces, namely, locally compact and separable metric spaces\cite{Barn88}.} in the sense that regardless of the initial distribution of $X_0$, the distributions $\mu_k$ of the random variables in the associated chain $\{X_k : k \geq 0\}$ converges weakly to the stationary distribution $\pi_\infty$, that is, $\mathbf{E}[f(X_k)] \rightarrow \mathbf{E}[f(X)] = \int_{\Real^n}f(x) d\pi_\infty$ for any bounded continuous function $f : \Real^n \rightarrow \Real$, where $X$ is the limiting (in distribution) random variable distributed as $\pi_\infty$. As a consequence, no matter what an initial distribution, the Markov process generated by RIA heads, with probability $1$, for the steady state expressed by the attractive distribution $\pi_\infty$, that is the IFS invariant measure.  This fact forms the basis for Elton's ergodic theorem (see \cite{Elto87}, Corollary 1), which states that for any $x \in \Real^n$ and $f$ defined above, the space average $\mathbf{E}[f(X)]$ is equal with probability $1$ to the time average of $\{ f(X_k) : X_0 = x, k \geq 0\}$, that is,
\begin{equation}\label{birkhoof}
    \lim_{m\rightarrow \infty} \frac{1}{m} \sum_{k=0}^{m-1} f(X_k) = \int_{\Real^n}f(x) d\pi_\infty \;\; a.s. 
\end{equation}
In other words, empirical probability measures $\mu_m = \tfrac{1}{m}\sum_{k=0}^{m-1}\delta_{x_k}$ of almost every\footnote{That is, except for a zero-measure set of exceptions.} orbit of the chain $\{X_k\}$ converge weakly to $\pi_\infty$. 

In effect, given a bounded continuous function $f : \Real^n \rightarrow \Real$, one can determine the right-hand side integral in the equation above by plugging an orbit generated by RIA into the left-hand side of the equation, and we are guaranteed that the result will be correct with probability one. In particular, given a Borel subset $A$, one can approximate the value of $\pi_\infty(A)$ by $\mu_m(A)$ for a certain large $m$ (under an additional technical requirement of the zero $\pi_\infty$-measure of $A$'s boundary), which technically is equivalent to counting and then normalizing the number of points that pop into $A$ in a finite-time realization ($m$ steps) of the algorithm. Such an approach leads to a popular method for visualizing the IFS invariant measure on a discrete grid of pixels. Below we investigate the possible effects of the application of RIA in its variant for visualizing invariant measures when the algorithm is applied in the context of a DIFS without any metric properties imposed on the mappings.  

As pointed out in Sec.~\ref{sec_statDist}, a stationary distribution of an \emph{irreducible} Markov chain is strictly positive and exists if and only if the chain is positive recurrent. This stationary distribution extends trivially to a stationary probability measure of any \emph{reducible} Markov chain whose state space includes the recurrent positive state space of this smaller irreducible chain as one of its closed communication classes (the measure of the complement of the closed communication class is zero). 
A positive recurrent communication class exists if the chain generated by a DIFS possesses a positive recurrent state. A positive recurrent state exists, first of all, if the state space of the associated chain is, or can be appropriately restricted to a finite set (cf. the second part of Theorem \ref{Markov_thm}---we will show that this holds for DIFSs arising from hyperbolic IFSs) or if some other, more general criteria for positive recurrence hold (e.g., the existence of a closed class satisfying Doeblin's criterion\cite{Stro05} or Foster's criterion\cite{Serf09}). 
Now, suppose that the associated chain has more than one class $\mathcal{A}_k$, $k\in I$, each being positive recurrent, so the chain has at least $|I| \in \Natural$ stationary distributions. But if $\pi_k$'s are stationary distributions of the chain, then any convex combination of these distributions is also a stationary distribution, because
\begin{equation}\label{uncountable_stationary}
    \Big(\sum_{k\in I} \alpha_k [\pi_k]\Big)P = \sum_{k\in I} \alpha_k ([\pi_k] P) = \sum_{k\in I} \alpha_k [\pi_k],
\end{equation}
where $P$ is the transition matrix of the chain, and $\sum_{k\in I}\alpha_k = 1$, $\alpha_k \geq 0$. 

Let us summarize our considerations on the DIFS stationary measures with the following:
\medskip
\begin{corollary}\label{DIFS_stat_measures}
Let $\{ S; \tilde{w}_1, \dots, \tilde{w}_N; p_1,\dots, p_N\}$ be a DIFS, and let $\mathcal{A}_{\mathcal{F}}$ be the family of the recurrent communication classes of the associated Markov chain. Denote by $\mathcal{A}^{+}_{\mathcal{F}}$ the positive recurrent subfamily of $\mathcal{A}_{\mathcal{F}}$. Then:\\
\noindent\emph{(a)} The Markov chain has a stationary probability measure if $|\mathcal{A}^{+}_{\mathcal{F}}| \geq 1$.\\ 
\noindent\emph{(b)} if $|\mathcal{A}^{+}_{\mathcal{F}}| = 1$, then the stationary probability measure is unique and supported by the single positive recurrent communication class in $\mathcal{A}^{+}_{\mathcal{F}}$.\\ 
\noindent\emph{(c)} If $|\mathcal{A}^{+}_{\mathcal{F}}| > 1$, then the chain has uncountable number of stationary probability measures of the form \eqref{uncountable_stationary}, that is, the stationary distributions are convex combinations of the stationary probability measures on the sets in $\mathcal{A}^{+}_{\mathcal{F}}$ and supported by their unions.      
\end{corollary}
\medskip
Obviously, by the definition of an attractive distribution, a stationary distribution can be globally attractive only if the stationary distribution is unique. Therefore, a necessary condition for a stationary probability measure $\pi$ supported by the sets in $\mathcal{A}^{+}_{\mathcal{F}}$ to be the globally attractive distribution is $|\mathcal{A}^{+}_{\mathcal{F}}| = 1$, that is, there has to be exactly one positive recurrent communication class, say $\mathcal{A}^{+}_1$, in the state space of the chain. However, in order to provide weak convergence of probability measures to $\pi$ independently of the initial distribution, in addition there cannot be other closed recurrent and closed transient classes in the state space but $\mathcal{A}^{+}_1$ and, moreover, the probability for the chain to escape from the transient states $\mathcal{T}$ has to be one. As long as the latter "escape-from-transient-class" condition is always met when the space upon which a DIFS acts is, or can be rightly restricted to a finite set, we usually cannot ensure that there will be only one recurrent communication class in the state space even in the case of DIFSs arising from hyperbolic IFSs (however, bearing in mind Theorem \ref{attractor_num_thm}, Theorem \ref{contr_half_thm} as well as the heuristic argument given in Sec.~\ref{on_multiattractor}). Therefore, even if a DIFS is a discretized version of hyperbolic IFS, it can have uncountable many stationary distributions (and thus no globally attractive distribution), and hence DIFSs are not in general contractive even on average (because in the latter case the stationary distribution is unique). 

Nonetheless, given a DIFS such that the associated chain leaves the transient class $\mathcal{T}$ with probability one, that is, $\prob(\exists i\in \Natural : \tilde{X}_i \notin \mathcal{T}) = 1$ independently of $\tilde{X}_0 \in \mathcal{T}$, we immediately get that the recurrent family $\mathcal{A}_{\mathcal{F}}$ is nonempty (since $\mathcal{T}$ is not a closed class, so there must be at least one recurrent state) and thus the orbit generated by RIA is guaranteed (almost surely) to reach one of the sets $\mathcal{A}_k$ in $\mathcal{A}_{\mathcal{F}}$. Since the sets are closed classes, the orbit will never leave such a set after it gets to the set. Moreover, due to Eq.~\eqref{recurrent_orbit_visit} it will visit all points $\tilde{x}$ in $\mathcal{A}_k$. Yet the orbit is not guaranteed to stabilize according to a stationary distribution even if the set is positive recurrent and, hence, the irreducible Markov chain on the set possesses a stationary distribution. The reason is that in order for a stationary distribution to be attractive for distributions on a positive recurrent communication class, the class should additionally contain an \emph{aperiodic} state, a state for which the greatest common divisor of possible numbers of steps to return (periods) is $1$ (which would imply that all states of the class were aperiodic too). It is also worth noting that in contrast to positive recurrence, aperiodicity is not guaranteed by finiteness of a state space. Therefore, even if a domain on which a DIFS acts is finite, there is no guarantee that the DIFS has a stationary probability measure that is attractive at the very least locally for distributions on a positive recurrent communication class. Anyway, any positive recurrent Markov chain is ergodic (see remark below), which means that despite the fact the associated chain within a positive recurrent communication class $\mathcal{A}_k^{+}$ does not necessarily reach a stationary distribution, the orbit generated by RIA can still be used to determine the value of the counterpart of the integral on the right-hand side of Eq.~\eqref{birkhoof}. In particular, if $\pi$ is a stationary distribution on $\mathcal{A}_k^{+}$, then for any state
$\tilde{y} \in \mathcal{A}_k^{+}$, 
\begin{equation*}
    \lim_{m\rightarrow \infty} \frac{1}{m} \sum_{i=0}^{m-1} \mathbf{1}_{\{\tilde{y}\}}(\tilde{X}_i) = \pi \{\tilde{y}\} \;\; a.s. 
\end{equation*}
provided that $\Pr(\tilde{X}_0 \in \mathcal{A}_k^{+}) = 1$. Hence, for any $\tilde{y} \in \mathcal{A}_k^{+}$,
\begin{equation*}
    \Pr\Big(\lim_{m\rightarrow \infty} \frac{1}{m} \sum_{i=0}^{m-1} \mathbf{1}_{\{\tilde{y}\}}(\tilde{X}_i) = \pi \{\tilde{y}\} | \tilde{X}_j \in \mathcal{A}_k^{+}\Big) = 1,
\end{equation*}
provided that there is a certain $j \in \Natural$ such that $\Pr(\tilde{X}_j \in \mathcal{A}_k^{+}) > 0$. As a consequence, if the orbit generated by RIA enters $\mathcal{A}_k^{+}$, one can use the orbit to render an image of the stationary distribution supported by the set. On the other hand, if the orbit does not enter a positive recurrent communication class, that is, it arrives either in $\mathcal{A}_k$ that is not positive recurrent or it stays in a transient class, then for any $\tilde{x} \in S$,
\begin{equation*}
    \lim_{m\rightarrow \infty} \frac{1}{m} \sum_{k=0}^{m-1} \mathbf{1}_{\{\tilde{x}\}}(X_k) = 0 \;\; a.s.
\end{equation*}
no matter what the initial distribution (see \cite{Stro05}, pp. 72--74). In such a case, an attempt to visualize a measure with RIA would yield a gradually vanishing image of a measure on points in $S \subset \mathcal{D}^n(\delta)$.  
\medskip

\begin{remark}[On egodicity of Markov chains]\label{rem_ergodic}
It is common in the literature on countable Markov chains that an ergodic chain is defined as an irreducible chain that is both positive recurrent and aperiodic, which is a sufficient condition for the chain to have an attractive (limiting) and, thus, unique stationary distribution. Such a chain is also ergodic in the sense of ergodic theory by some standard arguments\footnote{The arguments come down to expressing a Markov chain in terms of a measure preserving dynamical system on the space of the orbits of the chain, so that the $\sigma$-algebra is generated by the cylinders in the product space of the chain's state space and a map $T$ is the left shift operator preserving a measure $\mu$ being a counterpart of the stationary measure of the chain. Then, $T$ is ergodic (i.e., for every measurable $B$ such that $T^{-1}(B) = B$, $\mu(B) = 1$ or $0$) by the irreducibility of the state space of the chain. The Birkhoff's property of the almost sure convergence of the time averages on the orbits of an irreducible positive recurrent chain to the space average (as in Eq.~\eqref{birkhoof}) is the immediate consequence the ergodicity of the measure preserving operator $T$.} However, this usage of the term 'ergodic' in this context is somewhat misleading because it suggests that aperiodicity is necessary for a Markov chain to be ergodic in the sense of ergodic theory, whereas---as we pointed out earlier---every irreducible positive recurrent chain is ergodic (see e.g. \cite{Stro05,Walt00}, also a remark on Markov ergodicity by Elton in \cite{Elto87}). An instructive instance of an irreducible positive recurrent chain that is not aperiodic, but still ergodic is a 2-state chain with the transition matrix $\big[\begin{smallmatrix}
  0 & 1\\
  1 & 0
\end{smallmatrix}\big]$. 
The chain is positive recurrent and periodic with period $2$, it has a unique stationary distribution $[\tfrac{1}{2}, \tfrac{1}{2}]$, but the distribution is not attractive. Nevertheless, it is clearly seen that the chain is ergodic.
\end{remark}
\medskip
\begin{corollary}\label{Coro_DIFS_render}
Let $\{ S; \tilde{w}_1, \dots, \tilde{w}_N; p_1,\dots, p_N\}$ be a DIFS such that the associated Markov chain $\{\tilde{X}_i : i\geq 0\}$ satisfies $\prob(\exists i\in \Natural : \tilde{X}_i \notin \mathcal{T}) = 1$ for any $\tilde{X}_0 = \tilde{x} \in S$. Let $\{\tilde{x}_i\}_{i=0}^{\infty}$ be an orbit generated by RIA. Then, for a certain $m \in \Natural$, with probability one,
\begin{equation*}
\{\tilde{x}_i\}_{i=m}^{\infty} = \mathcal{A}_k, 
\end{equation*}
where $\mathcal{A}_k \in \mathcal{A}_{\mathcal{F}}$, that is, the set is one of the recurrent communication classes of $\{\tilde{X}_i\}$. 

In addition, if the set is positive recurrent, $\mathcal{A}_k \in \mathcal{A}_{\mathcal{F}}^{+}$, and $\Pr(\tilde{X}_j \in \mathcal{A}_k) = 1$ for a certain $j\in \Natural$, then, with probability one, for any $\tilde{y} \in S$, the ratio of the points in $\{\tilde{x}_i\}_{i=0}^{M}$ that coincide with $\tilde{y}$ to their total number $M$ converges to the value of $\pi_k\{\tilde{y}\}$ as $M \rightarrow \infty$, where $\pi_k$ is the stationary probability measure of the chain on $\mathcal{A}_k$; that is
\begin{equation*}
    \prob\Big(\lim_{M\rightarrow \infty}\frac{1}{M}\sum_{i=0}^{M-1} \delta_{\tilde{X}_i} = \pi_k | \tilde{X}_j \in \mathcal{A}_k \Big) = 1
\end{equation*}
and 
\begin{equation*}
    \prob\Big(\lim_{M\rightarrow \infty}\frac{1}{M}\sum_{i=0}^{M-1} \mathbf{1}_{\{ \tilde{y}\}}(\tilde{X}_i) = \pi_k\{\tilde{y}\} | \tilde{X}_j \in \mathcal{A}_k\Big) = 1,
\end{equation*}
if only $\Pr(\tilde{X}_j \in \mathcal{A}_k) > 0$ for a certain $j \in \Natural$. 

In particular, both statements of the corollary hold if $S$ is a finite set. 
\end{corollary}
\begin{figure}[t]
\centering
\begin{subfigure}{.3\textwidth}
  \centering
  \includegraphics[width=1\linewidth]{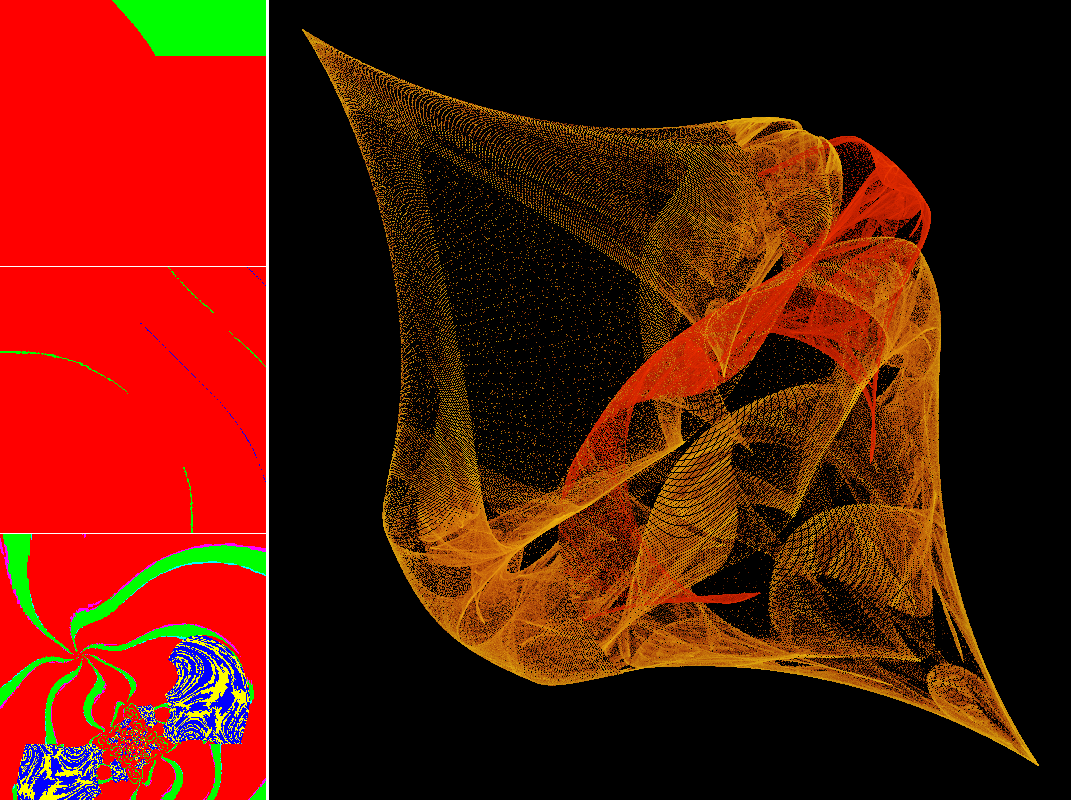}
  \caption{}
  \label{fig:7a}
\end{subfigure}
\begin{subfigure}{.3\textwidth}
  \centering
  \includegraphics[width=1\linewidth]{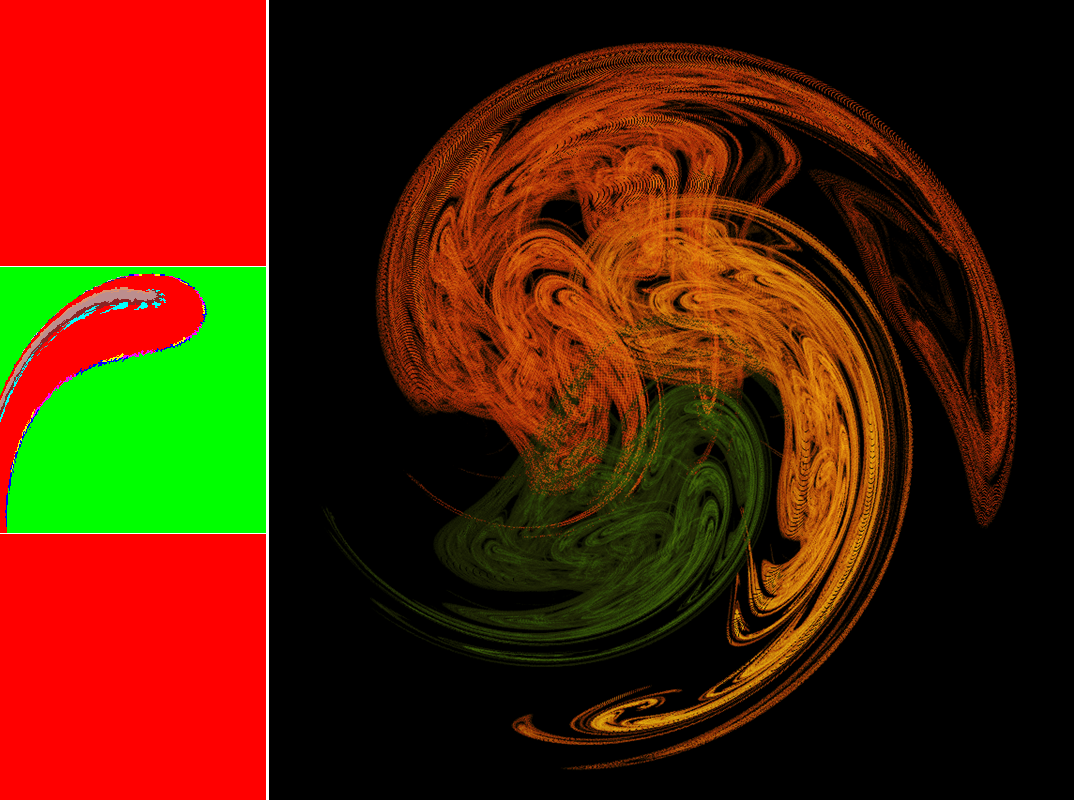}
  \caption{}
  \label{fig:7b}
\end{subfigure}
\begin{subfigure}{.3\textwidth}
  \centering
  \includegraphics[width=1\linewidth]{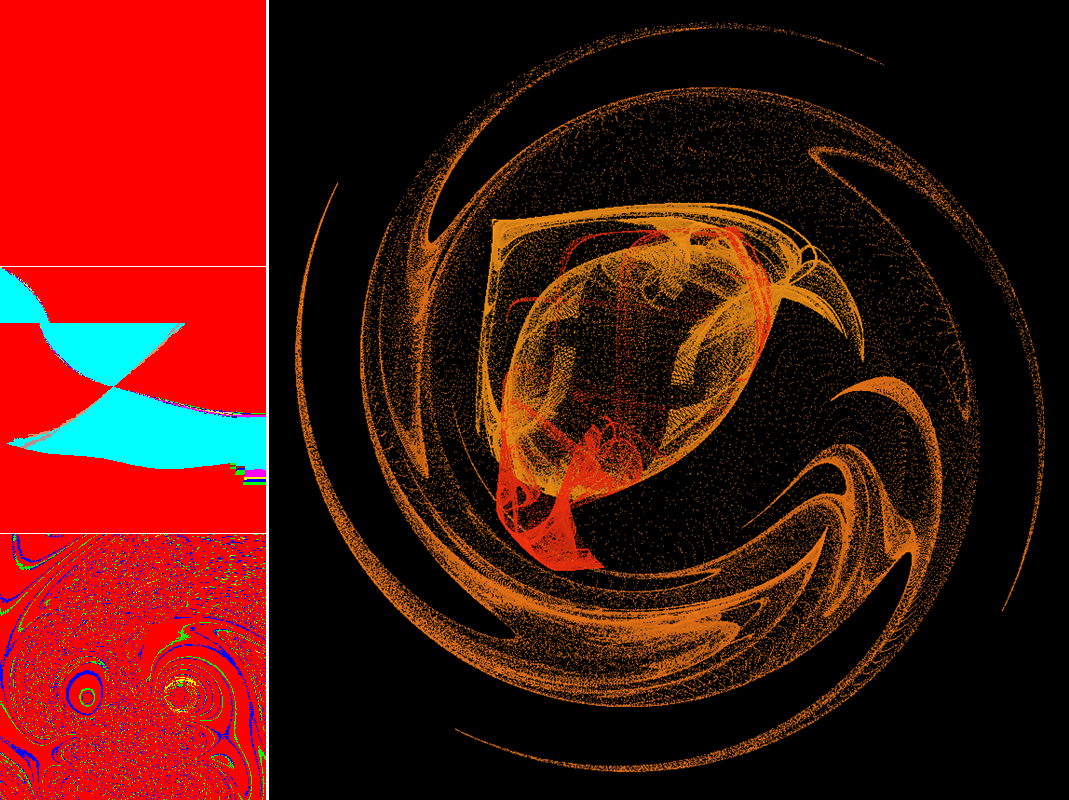}
  \caption{}
  \label{fig:7c}
\end{subfigure}
\caption{Examples of DIFS stationary distributions visualized with RIA. The DIFSs are composed of three maps. The number of the minimal absorbing set components of the maps are: a) 2, 3, 6, b) 1, 13, 1, c) 1, 11, 5. The basins of attractions of the maps are depicted on the left-hand side of the pictures}\label{fig:7}
\end{figure}
\medskip

In Fig.~\ref{fig:7} we present three examples of stationary probability measures of DIFSs defined on a squared subset $C \subset \mathcal{D}^2(\delta)$ of $1000\times1000$ resolution. Each DIFS consists of three mappings, which were constructed and stored in $1000\times1000$ arrays. Each array $W_i$ represented a single DIFS map $\tilde{w}_i$ in the form of a pair of integer numbers as $(W_i)_{k l} =\tilde{w}(k\delta, l\delta)/\delta$. In other words, each $W_i$ held information about a directed graph with vertices $(k,l)$ that are states of the associated Markov chain and edges generated by $\tilde{w}_i$, and thus the outdegree of each vertex being $1$. For each DIFS, the arrays were initialized with three $\delta$-roundoffs of contractive similarities describing a Sierpi\'nski's triangle, and then connections between the states were processed (changed) several times with the aid of nonlinear mappings, with care not to create a connection with a state lying outside $C$. Finally, the arrays (graphs) were additionally smoothed with a $3\times 3$ Gaussian filter. Analogously, the distributions of the DIFS place-dependent probabilities were stored in a $1000\times 1000$ array $D$ such that $(D)_{kl}$ was the distribution at state $(k, l)$. The array was initiated with uniform distributions, which were then perturbed with a nonlinear mapping. The renderings were obtained by means of RIA that generated an orbit in subset $C$ with the aid of the arrays $W_i$ and $D$, so that given state $(k, l)$, the next state was determined as $(W_i)_{kl}$ with $i$ drawn from the distribution $(D)_{kl}$. Since $W_i$'s and $D$ function as lookup tables, the orbit was generated in an extremely efficient manner, based only on fetches from the arrays and a pseudorandom generator. In the pictures the value of the measure of a state is interpreted as brightness, and the coloring reflects participation of a DIFS map in conveying a measure to a state, according to the formula $c_{new} = (c_{old} + c_i)/2$, where $c_{new}$ and $c_{old}$ are a new color and, respectively, a current color assigned to a visited state (with $c_{old}$ initialized with black at the beginning of the algorithm), and $c_i$ is a color assigned to the $i$th DIFS map.   

\subsection{DIFSs for hyperbolic IFSs}\label{sec_RIA_hyperbolic}
In this section we examine the case of DIFSs consisting of maps being $\delta$-roundoffs of contractions, in particular, DIFSs that are discretized versions of hyperbolic IFSs with constant probabilities. We start with the following theorem that establishes important facts concerning a DIFS composed of $\delta$-roundoffs of contractions.
\medskip
\begin{thmm}\label{DIFS_contr}
Let $\{\mathcal{D}^n(\delta);\tilde{w}_1, \dots, \tilde{w}_N; p_1, \dots, p_N \}$ be a DIFS such that the mappings $\tilde{w}_i$ are $\delta$-roundoffs of contractions $w_i$ on $(\Real^n, d)$. Let $o$ be any point of $\Real^n$. Denote $r_{max} := \max_i d(x_f^{(i)}, o)$, $\lambda_{max} := \max_i \lambda_i$ and $\alpha:= \frac{1 + \lambda_{max}}{1 - \lambda_{max}}$, where $x_f^{(i)}$ and $\lambda_i$ are the fixed points and, respectively, the contractivity factors of the mappings $w_i$. Then, for any $\varepsilon > 0$ and $S = B(o, r + \varepsilon) \cap \mathcal{D}^n(\delta)$, where $B(o, r)$ is an open ball in $(\Real^n, d)$, centred at $o$ and with radius $r = \alpha r_{max} + \theta(1-\lambda_{max})^{-1}$, the DIFS transformations map $S$ into itself, $\tilde{w}_i(S) \subset S$ for every $i\in \{1, \dots, N\}$. Moreover, the DIFS possesses nonempty $\mathcal{A}$, the set of all recurrent states of the associated Markov chain, and $\mathcal{A} \subset S$.
\end{thmm}
\begin{proof}
First we show that the DIFS transformations map $S$ into itself. Suppose that $\tilde{x}$ is any point of $S$. We need to show that, for any $i\in \{1,\dots, N \}$, $\tilde{w}_i(\tilde{x}) \in S$.Using the triangle inequality along with the Lipschitz continuity of $w_i$'s and the definition of a $\delta$-roundoff of a mapping (Def.~\ref{roundoff_def}), we get, for any $i\in \{ 1, \dots, N\}$ and $\varepsilon > 0$, that
\begin{equation*}
\begin{split}
    d(\tilde{w}_i&(\tilde{x}), o) \leq d(w_i(\tilde{x}), o) + d(\tilde{w}_i(\tilde{x}), w_i(\tilde{x}))
    \leq d(w_i(\tilde{x}), x_f^{(i)}) + d(x_f^{(i)}, o) + \theta
    \leq \lambda_{max}\; d(\tilde{x}, x_f^{(i)}) + r_{max} + \theta\\
    &\leq \lambda_{max}\; \big(d(\tilde{x}, o) + d(x_f^{(i)}, o) \big) + r_{max} + \theta
    < \lambda_{max}\; ( \alpha r_{max} + \theta(1-\lambda_{max})^{-1} + \varepsilon + r_{max} ) + r_{max} + \theta\\
    &= \alpha r_{max} + \lambda_{max} (\theta(1-\lambda_{max})^{-1} + \varepsilon)
    < \alpha r_{max} + \theta(1-\lambda_{max})^{-1} + \varepsilon
\end{split}
\end{equation*}
as required. 

Now we show that the Markov chain associated with the DIFS possesses a set $\mathcal{A}$ of recurrent states. By Theorem \ref{Markov2Hutchinson} (a), the set $S$ is a closed class and, moreover, $S$ is by definition finite, so due to the second part of Theorem \ref{Markov_thm}, the Markov chain possesses at least one nonempty recurrent communication class  within $S$, and hence $\mathcal{A} \neq \varnothing$. 

To complete the proof we need to show that $\mathcal{A} \subset S$. By the assumption of the theorem, $\tilde{w}_i$ are roundoffs of contractions, so by Corollary \ref{contraction_minimumset}, for each $\tilde{w}_i$, there is a finite minimal absorbing set $\mathcal{M}[\tilde{w}_i]$ (with respect to the whole space $\mathcal{D}^n(\delta)$). Now, fix some $i \in \{1, \dots, N \}$ and observe that if $\tilde{x} \in \mathcal{A}$, then $\tilde{x}$ is located in one of the basins of attraction of $\mathcal{M}[\tilde{w}_i]$. Therefore, for \emph{each} $\tilde{x} \in \mathcal{A}$, there is a state $\tilde{y} \in \mathcal{M}[\tilde{w}_i]$ accessible from $\tilde{x}$, and thus also $\tilde{y} \in \mathcal{A}$, because $\mathcal{A}$ is a closed class. In addition, $\mathcal{A}$ is composed of communication classes $\mathcal{A}_k$. From this we conclude that any set which is a closed class and, at the same time, includes $\mathcal{M}[\tilde{w}_i]$ has to contain $\mathcal{A}$ too. The set $S$ is a closed class, so to finish the proof it suffices to show that $S \supset \mathcal{M}[\tilde{w}_i]$. On the basis of Theorem \ref{thm_alpha_eps}, for every point $\tilde{x} \in \mathcal{M}[\tilde{w}_i]$, $d(\tilde{x}, x_f^{(i)}) \leq \theta(1-\lambda_i)^{-1}$, and hence for any $\varepsilon > 0$,
\begin{equation*}
\begin{split}
    d(\tilde{x}, o) \leq d(x_f^{(i)}, o) + d(\tilde{x}, x_f^{(i)})
    \leq r_{max} + \theta(1-\lambda_{max})^{-1}
    < \alpha r_{max} + \theta(1-\lambda_{max})^{-1} + \varepsilon, 
\end{split}
\end{equation*}
because $\alpha \geq 1$. This completes the proof. 
\end{proof}
\medskip
\begin{remark}\label{IFS_restr}
One can easily check that if $\{\Real^n; w_1, \dots, w_N; p_1, \dots, p_N \}$ is a hyperbolic IFS, then for any $\varepsilon > 0$ and $S = B(o, r + \varepsilon)$, where $r = \lim_{\delta\rightarrow 0}\alpha r_{max} + \theta(1-\lambda_{max})^{-1} = \alpha r_{max}$, the IFS transformations map $S$ into itself, $w_i(S) \subset S$ for every $i\in \{1, \dots, N\}$. Hence, the space upon which the IFS acts can be restricted to any of the compact sets (closed balls) $\overline{S}$, and the IFS attractor $A_\infty \subset\overline{S}$.
\end{remark}
\medskip
\begin{corollary}\label{coro_DIFS_contr}
If $\{ \mathcal{D}^n(\delta); \tilde{w}_1, \dots, \tilde{w}_N; p_1,\dots, p_N\}$ is a DIFS in which the mappings $\tilde{w}_i$ are $\delta$-roundoffs of contractions, then $\mathcal{A}$ is nonempty and finite and hence consists of positive recurrent states; thus $\mathcal{A}_\mathcal{F} = \mathcal{A}^{+}_\mathcal{F} \neq \varnothing$. Therefore, the DIFS possesses stationary probability measures supported by unions of the finite positive recurrent communication classes $\mathcal{A}_k^{+}$ of which $\mathcal{A}$ is composed. Moreover, for every $\tilde{X}_0 = \tilde{x} \in \mathcal{D}^n(\delta)$, $\prob(\exists i\in \Natural : \tilde{X}_i \notin \mathcal{T}) = 1$. Therefore, by Corollary \ref{Coro_DIFS_render}, each run of RIA results, with probability one, in rendering a stationary probability measure supported by one of the classes $\mathcal{A}_k^{+}$. The stationary probability measure is unique if for a certain $i\in \{1, \dots, N\}$, the minimal absorbing set for $\tilde{w}_i$ consists of a single component, or there are $i, j \in \{1, \dots, N\}$, such that $\mathcal{M}[\tilde{w}_i, S] \subset \mathcal{B}[\mathcal{M}_k[\tilde{w}_j, S]]$ for a certain component $\mathcal{M}_k[\tilde{w}_j]$ of the minimal absorbing set $\mathcal{M}_k[\tilde{w}_j]$ (Theorem \ref{attractor_num_thm}).    
\end{corollary}
\begin{proof}
The only thing to show is that for every $\tilde{X}_0 = \tilde{x} \in \mathcal{D}^n(\delta)$, $\prob(\exists i\in \Natural : \tilde{X}_i \notin \mathcal{T}) = 1$. Using Theorem \ref{DIFS_contr}, for any $\tilde{x} \in \mathcal{T}$ one can construct a finite closed class $S$ so that $\tilde{x}\in S$ and $\mathcal{A} \subset S$. Therefore, the "escape-from-transient-class" conclusion follows from the second part of Theorem \ref{Markov_thm}.
\end{proof}
\medskip
The lemma below establishes a connection between a hyperbolic IFS and a DIFS composed of $\delta$-roundoffs of the IFS mapping, in the form of an assertion on mutual shadowing of orbits of Markov chains generated by the DIFS and the corresponding IFS. 
\medskip
\begin{lem}\label{DIFS_shadowing_thm}
Let $\{ \mathcal{D}^n(\delta); \tilde{w}_1, \dots, \tilde{w}_N; p_1,\dots, p_N\}$ be a DIFS in which $\tilde{w}_i$'s are $\delta$-roundoffs of contractions $w_i$ on $(\Real^n, d)$. Let $\{\Real^n; w_1, \dots, w_n; q_1, \dots, q_N\}$ be an IFS with strictly positive probabilities $q_i \in (0, 1]$, $\sum_{i=1}^N q_i = 1$. Let $\tilde{x}_0$ be any point of $\mathcal{D}^n(\delta)$. Let $X = \{X_k : X_0 = \tilde{x}_0 \}$ and $\tilde{X} = \{\tilde{X}_k : \tilde{X}_0 = \tilde{x}_0 \}$ be the chains generated by the IFS and the DIFS, respectively. Then for any orbit $\{x_{k} = w_{i_k}(x_{k-1})\}$ of $X$ (respectively, any orbit $\{\tilde{x}_{k} = \tilde{w}_{i_k}(\tilde{x}_{k-1})\}$ of $\tilde{X}$), there exists an orbit $\{\tilde{x}_{k} = \tilde{w}_{i_k}(\tilde{x}_{k-1})\}$ of $\tilde{X}$ (respectively, an orbit $\{x_{k} = w_{i_k}(x_{k-1})\}$ of $X$) such that, for any $k \in \Natural$,
\begin{equation}\label{shadowing_ineq}
    d(x_k, \tilde{x}_k) \leq \theta(1-\lambda_{max})^{-1},
\end{equation}
where $\lambda_{max} = \max_i \lambda_i$, $\lambda_i$ is the contractivity factor of $w_i$.
\end{lem}
\begin{proof}
The mutual existence of the orbits $\{\tilde{x}_{k} = \tilde{w}_{i_k}(\tilde{x}_{k-1})\}$ and $\{x_{k} = w_{i_k}(x_{k-1})\}$ of the chain $\tilde{X}$ and $X$, respectively, is trivially provided by the strict positivity of probability functions $p_i(.)$ and probability weights $q_i$. Therefore, all we need to show is that all points of the orbits satisfy inequality \eqref{shadowing_ineq}. The proof is by induction.For $k=1$, we have
\begin{equation*}
\begin{split}
    d(x_1, \tilde{x}_1) &\leq d(x_1, w_{i_1}(\tilde{x}_0) ) + d(w_{i_1}(\tilde{x}_0), \tilde{x}_1)
    \leq d(w_{i_1}(\tilde{x}_0), \tilde{w}_{i_1}(\tilde{x}_0)) \leq \theta,
\end{split}
\end{equation*}
where the last inequality follows from the definition of a $\delta$-roundoff of a mapping (Def.~\ref{roundoff_def}). Now assume that inequality \eqref{shadowing_ineq} is true for $k$. Then
\begin{equation*}
\begin{split}
    d(x_{k+1}, \tilde{x}_{k+1}) &\leq d(x_{k+1}, w_{i_{k+1}}(\tilde{x}_k)) + d(w_{i_{k+1}}(\tilde{x}_k), \tilde{x}_{k+1})\\
    &\leq d(w_{i_{k+1}}(x_{k}), w_{i_{k+1}}(\tilde{x}_k)) + d(w_{i_{k+1}}(\tilde{x}_k), \tilde{w}_{i_{k+1}}(\tilde{x}_{k}))\\
    &\leq \lambda_{max}\; d(x_k, \tilde{x}_k) + \theta
    \leq \lambda_{max}\; (1-\lambda_{max})^{-1} \theta + \theta = \theta (1-\lambda_{max})^{-1}
\end{split}
\end{equation*}
as required.
\end{proof}
\medskip
In turn, the next theorem answers the question about geometrical resemblance between DIFS invariant sets $\mathcal{A}_k^{+} \in \mathcal{A}_{\mathcal{F}}$ and the attractor $A_{\infty}$ of a hyperbolic IFS, expressed in terms of the Hausdorff distance.  
\medskip
\begin{thmm}
Let $\{ \mathcal{D}^n(\delta); \tilde{w}_1, \dots, \tilde{w}_N; p_1,\dots, p_N\}$ be a DIFS in which $\tilde{w}_i$'s are $\delta$-roundoffs of contractions $w_i$ on $(\Real^n, d)$, and let $\{\Real^n; w_1, \dots, w_n; q_1, \dots, q_N\}$ be a corresponding hyperbolic IFS with strictly positive probabilities $q_i$. Let $\mathcal{A}_{\mathcal{F}}$ be the family of all positive recurrent classes of the Markov chain associated with the DIFS, and let $A_{\infty}$ be the attractor of the IFS. Then, independently of the values of probabilities $p_i(.)$ and $q_i$, for each $\mathcal{A}_k^{+} \in \mathcal{A}_{\mathcal{F}}$, the Hausdorff distance between $\mathcal{A}_k^{+}$ and $A_{\infty}$ is bounded from above as
\begin{equation}\label{hausdorff_distance}
    h(\mathcal{A}_k^{+}, A_{\infty}) \leq \theta(1-\lambda_{max})^{-1}.
\end{equation}
\end{thmm}
\begin{proof}
The Hausdorff distance between $\mathcal{A}_k^{+}$ and $A_{\infty}$ is $h(\mathcal{A}_k^{+}, A_{\infty}) = \inf\{\varepsilon \geq 0 : A_{\infty}\subset N(\mathcal{A}_k^{+}, \varepsilon) \textit{ and } \mathcal{A}_k^{+}\subset N(A_{\infty}, \varepsilon)\}$, where $N(A, \varepsilon) := \{x \in \Real^n : d(x, a) < \varepsilon,\; a \in A \}$ denotes the (open) $\varepsilon$-neighbourhood of a set $A$. First we show that for any $\varepsilon > 0$, $A_{\infty}\subset N(\mathcal{A}_k^{+}, \theta (1-\lambda_{max})^{-1}+ \varepsilon)$. We need to show that for any $a \in A_{\infty}$, there is a certain $\tilde{x} \in \mathcal{A}_k^{+}$ such that $d(a, \tilde{x}) < \theta (1-\lambda_{max})^{-1}+ \varepsilon$. Let $a$ be any point of $A_{\infty}$. The attractor is the support of the IFS invariant measure $\pi$, so for any $\varepsilon > 0$, $\pi(B(a, \varepsilon)) > 0$. Moreover, by Elton's ergodic theorem (Eq.~\eqref{birkhoof}) $\pi$ is ergodic, and hence, for any initial point $x_0 \in \Real^n$, almost every orbit $\{x_i\}_{i=0}^{\infty}$ of the Markov chain generated by the IFS visits $B(a, \varepsilon)$ infinitely often. Hence, for any $x_0 \in \Real^n$, there is a finite sequence of indices $i_m, \dots, i_1 \in \{1, \dots, N\}$ such that $d(w_{i_m}\circ \dots \circ w_1(x_0), a) < \varepsilon$. On that basis, putting $x_0 \in \mathcal{A}_k^{+}$ and using the previous lemma, we get that there is a finite sequence $\mathbf{i} = (i_m, \dots, i_1)$ of indices such that
\begin{equation*}
\begin{split}
    d(\tilde{w}_{\mathbf{i}}(x_0), a) &\leq d(\tilde{w}_{\mathbf{i}}(x_0), w_{\mathbf{i}}(x_0)) + d(w_{\mathbf{i}}(x_0), a)
    < \theta (1-\lambda_{max})^{-1} + \varepsilon,
\end{split}
\end{equation*}
where $w_{\mathbf{i}}(.) := w_{i_m}\circ \dots \circ w_1(.)$. But $\mathcal{A}_k^{+}$ is a closed class, and thus $\tilde{w}_{\mathbf{i}}(x_0) \in \mathcal{A}_k^{+}$. Hence, $A_{\infty}\subset N(\mathcal{A}_k^{+}, \theta (1-\lambda_{max})^{-1}+ \varepsilon)$ as required. 

Now we show that for any $\varepsilon > 0$, $\mathcal{A}_k^{+}\subset N(A_{\infty}, \theta (1-\lambda_{max})^{-1}+ \varepsilon)$. Let $\tilde{x}$ be any point of $\mathcal{A}_k^{+}$. Since $\mathcal{A}_k^{+}$ is a recurrent class, $\tilde{x}$ is a recurrent state and thus there is a finite sequence of indices $\mathbf{i} = (i_m, \dots, i_1) \in \{1, \dots, N\}^m$ such that $\tilde{w}_{\mathbf{i}}(\tilde{x}) = \tilde{x}$, and hence for any $j \in \Natural$, $\tilde{w}_{\mathbf{i}}^{\circ j}(\tilde{x}) = \tilde{x}$. Now let $a \in A_{\infty}$. Since the IFS maps $w_i$ are contractions, we have, for any $j\in \Natural$, 
\begin{equation*}
d(w_{\mathbf{i}}^{\circ j}(\tilde{x}), w_{\mathbf{i}}^{\circ j}(a)) \leq \lambda_{max}^{m \cdot j}\; d(\tilde{x}, a)
\end{equation*}
and hence for any $\varepsilon > 0$, there is $M\in \Natural$ such that
\begin{equation*}
    d(w_{\mathbf{i}}^{\circ M}(\tilde{x}), w_{\mathbf{i}}^{\circ M}(a)) \leq \lambda_{max}^M\; d(\tilde{x}, a) < \varepsilon,
\end{equation*}
because $\lambda_{max} \in [0, 1)$. On the basis of above, using the previous lemma we conclude that
\begin{equation*}
\begin{split}
    d(\tilde{x}, w_{\mathbf{i}}^{\circ M}(a)) &= d(\tilde{w}_{\mathbf{i}}^{\circ M}(\tilde{x}), w_{\mathbf{i}}^{\circ M}(a))\\
    &\leq d(\tilde{w}_{\mathbf{i}}^{\circ M}(\tilde{x}), w_{\mathbf{i}}^{\circ M}(\tilde{x}))+ d(w_{\mathbf{i}}^{\circ M}(\tilde{x}), w_{\mathbf{i}}^{\circ M}(\tilde{a}))\\
    &< \theta(1-\lambda_{max})^{-1} + \varepsilon.
\end{split}
\end{equation*}
But $w_{\mathbf{i}}^{\circ M}(a) \in A_{\infty}$, because $w_i$'s map $A_{\infty}$ into itself. It follows that $\mathcal{A}_k^{+}\subset N(A_{\infty}, \theta (1-\lambda_{max})^{-1}+ \varepsilon)$ as required. 

Since both $A_{\infty}\subset N(\mathcal{A}_k^{+}, \theta (1-\lambda_{max})^{-1}+ \varepsilon)$ and $\mathcal{A}_k^{+}\subset N(A_{\infty}, \theta (1-\lambda_{max})^{-1}+ \varepsilon)$ for any $\varepsilon > 0$, we get that the infimum in the Hausdorff distance $h(\mathcal{A}_k^{+}, A_{\infty})$ is bounded from above by the value of $\theta (1-\lambda_{max})^{-1}$. This completes the proof. 
\end{proof}
\medskip
\begin{corollary}
Let $\{\Real^n; w_1, \dots, w_n; p_1, \dots, p_N\}$ be a hyperbolic IFS with the attractor $A_{\infty}$. Let $\{ \mathcal{D}^n(\delta); \tilde{w}_1, \dots, \tilde{w}_N; p_1,\dots, p_N\}$ be the corresponding DIFSs (with the same constant probabilities as in the IFS), parametrized by $\delta >0$, in which $\tilde{w}_i$'s are $\delta$-roundoffs of contractions $w_i$. Let $\mathcal{A}_{\mathcal{F}}(\delta) = \{\mathcal{A}_k^{+}(\delta)\}_k$ be the family of all positive recurrent classes of the Markov chain associated with a DIFS for fixed $\delta$.  Then
\begin{equation*}
    \lim_{\delta\rightarrow 0} \mathcal{A}_k^{+}(\delta) = A_{\infty}\;\; \textit{(in the Hausdorff metric)},
\end{equation*}
where $\mathcal{A}_k^{+}(\delta)$ is any set from $\mathcal{A}_{\mathcal{F}}(\delta)$ for fixed $\delta$. \end{corollary}
\begin{proof}
By Theorem \ref{DIFS_contr}, for any $\delta > 0$, the set of all recurrent states of the associated Markov chain is nonempty and finite (and thus compact), and so are the set's subsets $\mathcal{A}_k^{+}(\delta)$. Therefore, for any $\delta > 0$, $\mathcal{A}_k^{+}(\delta)$ is an element of $\mathcal{H}(\Real^n)$, the family of all nonempty and compact subsets of $\Real^n$. By a standard argument for hyperbolic IFSs, $A_{\infty} \in \mathcal{H}(\Real^n)$ too. Since $\diam_d(C_\delta) \rightarrow 0$ as $\delta \rightarrow 0$, and the Hausdorff distance $h$ is a metric on $\mathcal{H}(\Real^n)$, the conclusion follows from inequality \eqref{hausdorff_distance}.      
\end{proof}
\medskip
We also have a theorem concerning a relationship between DIFS and IFS measures: 
\medskip
\begin{thmm}\label{thm_conv2invmeasure}
Let $\{\Real^n; w_1, \dots, w_n; p_1, \dots, p_N\}$ be a hyperbolic IFS, and let $\pi_{\infty}$ be the IFS invariant measure. Let $\{ \mathcal{D}^n(\delta); \tilde{w}_1, \dots, \tilde{w}_N; p_1,\dots, p_N\}$ be the corresponding DIFSs (with the same constant probabilities as in the IFS), parametrized by $\delta > 0$, in which $\tilde{w}_i$'s are $\delta$-roundoffs of contractions $w_i$. Let $\mathcal{A}_{\mathcal{F}}(\delta)$ be the family of all positive recurrent classes of the Markov chain associated with a DIFS for fixed $\delta$, and $\mathit{\Pi}(\delta) = \{\pi_k(\delta) : \supp(\pi_k(\delta)) \in \mathcal{A}_{\mathcal{F}}(\delta)\}$ be the family of the chain's stationary distributions supported by the sets in $\mathcal{A}_{\mathcal{F}}(\delta)$. Then for any continuous and bounded $f : \Real^n \rightarrow \Real$, 
\begin{equation}\label{weak_conv}
    \lim_{\delta\rightarrow 0} \sum_{\tilde{x} \in \mathcal{D}^n(\delta)} f(\tilde{x}) \pi_k(\delta)\{\tilde{x}\} = \int_{\Real^n} f(x) d\pi_{\infty},
\end{equation}
where $\pi_k(\delta)$ is any stationary distribution from $\mathit{\Pi}(\delta)$ for fixed $\delta$. In other words, $\pi_k(\delta)$'s converge weakly to $\pi_{\infty}$ as $\delta \rightarrow 0$. 
\end{thmm}
\begin{proof}
In \cite{Peru93} Peruggia showed that a very similar conclusion holds under the assumption that the fixed point of one of the IFS mappings coincides with $\mathbf{0}\in \mathcal{D}^2(\delta)$, the zero vector of the discrete (pixel) space, which naturally stays intact while changing the value of the discretization parameter $\delta$ (cf. \cite{Peru93}, Theorem 4.38, pp.~129--131). Although our theorem does not impose such a restriction, the proof is founded on similar arguments as those used in the proof by Peruggia.  

First, observe that the summation on the left hand side of Eq.~\eqref{weak_conv} can be restricted to the support of the measure $\pi_k(\delta)$, $\supp(\pi_k(\delta)) = \mathcal{A}_k^{+}(\delta)$, and, as we pointed out earlier, the corresponding Markov chain $\{\tilde{X}_i^{\delta}(\tilde{x}_0) : \tilde{X}_0^{\delta}(\tilde{x}_0) = \tilde{x}_0\in \mathcal{A}_k^{+}(\delta)\}$ generated by the DIFS (for fixed $\delta$) is (Birkhoff's) ergodic on $\mathcal{A}_k^{+}(\delta)$. In addition, by Elton's ergodic theorem, the IFS invariant measure $\pi_{\infty}$ is ergodic for the Markov chain $\{ X_i(x_0) : X_0 = x_0 \in \Real^n \}$ generated by the IFS on $\Real^n$. Putting these facts together, we get that for any $\delta >0$, with probability one, 
\begin{equation}\label{weak2ergodic}
\begin{split}
    \Big| \sum_{\tilde{x} \in \mathcal{A}_k^{+}(\delta)} &f(\tilde{x}) \pi_k(\delta)\{\tilde{x}\} - \int_{\Real^n} f(x) d\pi_{\infty} \Big|\\
    &= \lim_{m\rightarrow \infty} \frac{1}{m} \Big| \sum_{i=0}^{m-1} \big(f(\tilde{X}^{\delta}_i(\tilde{x}_0)) - f(X_i(\tilde{x}_0))\big)
    \Big|\\
    &\leq \lim_{m\rightarrow \infty} \frac{1}{m} \sum_{i=0}^{m-1} \big|f(\tilde{X}^{\delta}_i(\tilde{x}_0)) - f(X_i(\tilde{x}_0))\big|
\end{split}
\end{equation}
where $\tilde{x}_0 \in \mathcal{A}_k^{+}(\delta)$. But $\tilde{X}_i^{\delta}(\tilde{x}_0) = \tilde{w}_{I_i}\big(\tilde{X}_{i-1}^{\delta}(\tilde{x}_0)\big)$ and $X_i(\tilde{x}_0) = w_{I_i}\big(X_{i-1}(\tilde{x}_0)\big)$, that is, both Markov chains are driven by the same sequence $\{I_i\}_{i\in \Natural}$ of the i.i.d.~random variables $I_i$ distributed as $[p_1, \dots, p_N]$. Hence, from Lemma \ref{DIFS_shadowing_thm}, 
\begin{equation}\label{rand_bound}
d(X_i(\tilde{x}_0), \tilde{X}^{\delta}_i(\tilde{x}_0)) \leq \theta(1 - \lambda_{max})^{-1}     
\end{equation}
for every $i \in \Natural$, where $\lambda_{max}$ is the maximum contractivity factor of the IFS mappings $w_i$. Now crucial is the observation, which will be shown to be true at the end of the proof, that there exists a compact set $E \subset \Real^n$ such that, for every $\delta \in (0, R)$, where $R > 0$ is a certain real number, $E \supset \{ X_i(\tilde{x}_0)\}$ and $E \supset \{ \tilde{X}^{\delta}_i(\tilde{x}_0)\}$ for any $ \tilde{x}_0\in \mathcal{A}_k^{+}(\delta)$, that is, none of the chains for $\delta \in (0, R)$ moves out of $E$. Then, on the basis of the Heine–Cantor theorem, $f$'s are uniformly continuous on $E$. Since the right-hand side of inequality \eqref{rand_bound} converges to $0$ as $\delta \rightarrow 0$, from this we conclude that for any $\varepsilon >0$, there is $\delta(\varepsilon) \in (0, R)$ such that for any $\delta \in (0, \delta(\varepsilon))$, $\big|f(\tilde{X}^{\delta}_i(\tilde{x}_0)) - f(X_i(\tilde{x}_0))\big| < \varepsilon$ for all $i \in \Natural$. Therefore, 
\begin{equation*}
    \lim_{\delta\rightarrow 0}\Big(\lim_{m\rightarrow \infty} \frac{1}{m} \sum_{i=0}^{m-1} \big|f(\tilde{X}^{\delta}_i(\tilde{x}_0)) - f(X_i(\tilde{x}_0))\big|\Big) = 0
\end{equation*}
and hence, taking limits as $\delta \rightarrow 0$ on both sides of inequality \eqref{weak2ergodic}, we get 
\begin{equation*}
\lim_{\delta\rightarrow 0} \sum_{\tilde{x} \in \mathcal{A}_k^{+}(\delta)} f(\tilde{x}) \pi_k(\delta)\{\tilde{x}\} = \int_{\Real^n} f(x) d\pi_{\infty}. \end{equation*}
Since the sets $\mathcal{A}_k^{+}(\delta)$ are the supports of the measures $\pi_k(\delta)$, the summation on the left-hand side of the above formula equals the summation over the whole space $\mathcal{D}^n(\delta)$, and thus we have arrived at the conclusion of the theorem. 

The remaining thing to show is the existence of a compact set $E$, in which the Markov chains $\{ X_i(\tilde{x}_0)\}$ and $\{ \tilde{X}^{\delta}_i(\tilde{x}_0)\}$ reside for any $\tilde{x}_0\in \mathcal{A}_k^{+}(\delta)$ and $\delta \in (0, R)$, so as to assure the uniform continuity of the functions $f$. To this end, we can apply the following construction: Fix $R > 0$ and observe that by inequality \eqref{hausdorff_distance}, for any $\delta \in (0, R)$, $\mathcal{A}_k^{+}(\delta) \subset N(A_{\infty}, r_0)$, $r_0 = \tfrac{1}{2}\diam_d(C_R)(1-\lambda_{max})^{-1}$, and because $A_{\infty}$ is compact, so is the closure $\overline{N}(A_{\infty}, r_0)$ of the neighbourhood. Hence, for all $\delta \in (0, R)$ and $\tilde{x}_0 \in \mathcal{A}_k^{+}(\delta)$, all Markov chains $\{ \tilde{X}^{\delta}_i(\tilde{x}_0)\}$ reside in $\overline{N}(A_{\infty}, r_0)$ (because $\mathcal{A}_k^{+}(\delta)$'s are closed classes). Next we extend $\overline{N}(A_{\infty}, r_0)$ so as to additionally encompass all Markov chains $\{ X_i(\tilde{x}_0)\}$ for $\delta \in (0, R)$ and $\tilde{x}_0 \in \mathcal{A}_k^{+}(\delta)$. Due to inequality \eqref{rand_bound}, it is easily done by doubling the radius of $\overline{N}(A_{\infty}, r_0)$, so the required compact set is $\overline{N}(A_{\infty}, 2 r_0)$. This completes the proof.  
\end{proof}

\section{Open problems}\label{sec_conclude}
First of all, an open problem is the limiting behavior of the discrete version of the Hutchison operator associated with a DIFS. Note that although the operator works in an entirely deterministic way, it can be viewed as generating all possible realizations of a Markov chain started from a give state. Therefore it is possible to analyse the behavior of the operator using the theory presented in this paper. This is the topic that we are going to investigate in the near future. 

As for DIFSs treated as discretized versions of IFSs, in this paper we tackled only the problem of discretizing IFSs that are hyperbolic. An open problem still remains discretization of the IFSs contractive on average with place-dependent probabilities. It seems that the form of DIFSs with place-dependent probabilities, with which we dealt in this paper, may happen to be insufficient to embrace discretized versions of the IFSs contractive on average because of the problem concerning the place-dependent distributions within a $\delta$-cube which now need to be somehow reduced to a single representative distribution located in the cube's center. Note that there are two possible approaches to discretizing such an IFS: the first method somehow determines a single distribution on the basis of all probability distributions within a $\delta$-cube and this distribution remains unchanged during the Markov chain evolution, and the second one in which the single distribution changes depending on the current distribution ascribed to a point that the chain visits when arrives at the $\delta$-cube at a given instant. While the first approach is consistent with a time-homogeneous Markov chain and the definition of DIFS that we admit in this paper, the second one leads to a time inhomogenous Markov chain on a discrete space, that is, the chain whose transition matrix changes in time. Ergodic properties of Markov chains generated by IFSs with time-dependent probabilities on compact spaces were studied by Stenflo in \cite{Sten98}. Nonetheless, the paper dealt with the probabilities converging in time to some limiting values and, to our knowledge, this is the only paper that tackles the problem of IFSs with time-dependent probabilities. As a matter of fact, the subject of the IFSs of this kind has not been recognized in the relevant literature, probably because of its seeming exoticness. It turns out, however, that there is nothing exotic in IFSs with time-dependent probabilities because they can arise quite naturally in real-world implementations of IFSs with place-dependent probabilities. Our conjecture on this second method of discretization of an average contractive IFS with place-dependent probabilities is that the corresponding DIFS with time-and-place dependent probabilities generates a time inhomogeneous Markov chain such that, with probability one, empirical distributions of its orbit converges weakly to a probability measure (i.e., in the spirit of Eq.~\eqref{birkhoof}), which is not a stationary distribution for the process itself (supposedly the process has no stationary distribution in general), but nevertheless the measure approximates the invariant measure of the original IFS with an error not exceeding the value of $\theta/(1-c)$ with respect to the Monge-Kantorovich metric, where $c < 1$ is the on average contractivity factor of the IFS. This is our working hypothesis which we would like to investigate in the future.  


\end{document}